\documentclass[a4paper, 11pt]{amsart}

\title{Perverse schobers and Orlov equivalences}
\author{Naoki Koseki and Genki Ouchi}
\date{}

\address{The University of Liverpool, Mathematical Sciences Building, Liverpool, L69 7ZL, UK.}
\email{koseki@liverpool.ac.uk}
\address{Graduate School of Mathematics, Nagoya University, Furocho, Chikusaku, Nagoya, Japan, 464-8602}
\email{genki.ouchi@math.nagoya-u.ac.jp}

\usepackage{amsmath, amssymb, amsthm, amscd, comment, mathtools, color}
\usepackage[frame,cmtip,curve,arrow,matrix,line,graph]{xy}
\usepackage{tikz}
\usepackage{todonotes}
\usetikzlibrary{intersections, calc}

\makeatletter
 
  \@addtoreset{equation}{section}
 \makeatother

\theoremstyle{plain}
\newtheorem{thm}{Theorem}[section]
\newtheorem{prop}[thm]{Proposition}
\newtheorem{def-prop}[thm]{Definition-Proposition}
\newtheorem{def-thm}[thm]{Definition-Theorem}
\newtheorem{def-lem}[thm]{Definition-Lemma}
\newtheorem{lem}[thm]{Lemma}
\newtheorem{cor}[thm]{Corollary}
\newtheorem*{thm*}{Theorem}

\theoremstyle{definition}
\newtheorem{defin}[thm]{Definition}

\newtheorem*{NaC}{Notation and Convention}
\newtheorem*{ACK}{Acknowledgement}

\theoremstyle{remark}
\newtheorem{rmk}[thm]{Remark}

\newtheorem{ex}[thm]{Example}

\DeclareMathOperator{\rk}{rk}

\DeclareMathOperator{\Spec}{Spec}
\DeclareMathOperator{\id}{id}
\newcommand{\dR}{\mathbf{R}}
\newcommand{\dL}{\mathbf{L}}
\newcommand{\bP}{\mathbb{P}}
\newcommand{\bC}{\mathbb{C}}
\newcommand{\bR}{\mathbb{R}}
\newcommand{\bQ}{\mathbb{Q}}
\newcommand{\bZ}{\mathbb{Z}}
\newcommand{\mcA}{\mathcal{A}}
\newcommand{\mcB}{\mathcal{B}}
\newcommand{\mcC}{\mathcal{C}}
\newcommand{\mcD}{\mathcal{D}}
\newcommand{\mcE}{\mathcal{E}}

\newcommand{\mcG}{\mathcal{G}}
\newcommand{\mcH}{\mathcal{H}}

\newcommand{\mcL}{\mathcal{L}}

\newcommand{\mcO}{\mathcal{O}}
\newcommand{\mcP}{\mathcal{P}}

\newcommand{\mcR}{\mathcal{R}}

\newcommand{\fK}{{\mathfrak K}}
\newcommand{\fL}{{\mathfrak L}}

\newcommand{\fP}{{\mathfrak P}}

\newcommand{\fS}{{\mathfrak S}}

\newcommand{\abs}{\mathrm{abs}}
\newcommand{\Dabs}{\mathrm{D}^{\mathrm{abs}}}

\DeclareMathOperator{\Hom}{Hom}

\DeclareMathOperator{\Coh}{Coh}

\DeclareMathOperator{\Sym}{Sym}
\DeclareMathOperator{\IC}{IC}

\DeclareMathOperator{\Perv}{Perv}

\DeclareMathOperator{\Cone}{Cone}

\DeclareMathOperator{\pt}{pt}

\DeclareMathOperator{\ST}{ST}

\DeclareMathOperator{\Br}{Br}

\DeclareMathOperator{\MF}{MF^{gr}}
\DeclareMathOperator{\HMF}{HMF^{gr}}

\DeclareMathOperator{\vect}{vect}

\DeclareMathOperator{\acyc}{acyc}
\DeclareMathOperator{\Fact}{Fact}
\DeclareMathOperator{\fact}{fact}

\DeclareMathOperator{\dgcat}{dgcat}
\DeclareMathOperator{\hodg}{hodgcat}
\DeclareMathOperator{\Mod}{Mod}
\DeclareMathOperator{\perf}{perf}
\DeclareMathOperator{\tri}{tri}
\DeclareMathOperator{\op}{op}

\DeclareMathOperator{\Isom}{Isom}
\DeclareMathOperator{\res}{res}
\DeclareMathOperator{\Auteq}{Auteq}
\DeclareMathOperator{\Aut}{Aut}
\DeclareMathOperator{\num}{num}
\DeclareMathOperator{\Fuk}{Fuk}
\DeclareMathOperator{\mirror}{mirror}

\begin{document}

\maketitle

\begin{abstract}
A perverse schober is a categorification 
of a perverse sheaf proposed by Kapranov--Schechtman. 
In this paper, we construct examples of perverse schobers 
on the Riemann sphere, 
which categorify the intersection complexes of 
natural local systems 
arising from the mirror symmetry for Calabi-Yau hypersurfaces. 
The Orlov equivalence plays a key role for the construction. 
\end{abstract}

\setcounter{tocdepth}{1}
\tableofcontents

\section{Introduction} \label{sec:intro}

\subsection{Motivation and Results}
A {\it perverse schober} is a conjectural categorification 
of a perverse sheaf introduced in a seminal paper \cite{ks15} 
by Kapranov--Schechtman. 
There is a well-known way to define 
the notion of local systems of categories, 
which are the simplest examples of perverse schobers. 
However, it is not clear how to define 
perverse sheaves of categories in general. 
At this moment, a general definition is only available 
on Riemann surfaces 
and on affine spaces stratified by hyperplane arrangements 
\cite{dkss21, ks15}. 
A key observation for the categorification 
in the case of Riemann surfaces 
is the classical result of 
\cite{bei87, ggm85, gmv96} that describes 
the category of perverse sheaves on a disk 
in terms of quiver representations. 
Quiver representations consist of linear algebraic data, 
and hence we can categorify them by replacing 
vector spaces and linear maps 
with categories and functors, respectively. 
More precisely, we will use the notion of 
{\it spherical functors} between dg categories 
developed in \cite{al17}. 
A lot of interesting classes of perverse schobers 
have been constructed via 
birational geometry \cite{bks18, don19a, don19b, svdb19, svdb20} 
and symplectic geometry \cite{dk21, ks15, kss20}. 

In this paper, we construct new examples of perverse schobers 
on the Riemann sphere $\bP^1$, 
which arise from mirror symmetry 
of Calabi-Yau hypersurfaces. 
Let $X \subset \bP^{n+1}$ be 
a smooth hypersurface of degree $n+2$, 
which is a Calabi-Yau variety of dimension $n$. 
Then there exists the mirror family of $X$ over 
$\bP^1 \setminus \{1, \infty \}$, 
with the unique orbifold point at $0 \in \bP^1$. 
By a variation of complex structures, 
we obtain a natural homomorphism 
\[
\pi_1(\bP^1\setminus \{0, 1, \infty\}) \to 
\Auteq(D^\pi\Fuk(X^\vee)), 
\]
where $X^\vee$ denotes the mirror of $X$. 
Applying the conjectural mirror symmetry equivalence 
$D^\pi\Fuk(X^\vee) \simeq D^b(X)$ 
(partly proved in \cite{pz98, sei15, she15}), 
we get a homomorphism 
\begin{equation} \label{eq:introLS}
\pi_1(\bP^1\setminus \{0, 1, \infty\}) \to \Auteq(D^b(X)), 
\end{equation}
which we think of as a local system of categories 
with the fiber $D^b(X)$. 
It is proved (cf. \cite{bh06, cir14, hor99}) 
that the morphism (\ref{eq:introLS}) 
maps the simple loops around the points 
$\infty, 1$ to the autoequivalences 
$(-) \otimes \mcO_X(1)$ and $\ST_{\mcO_X}$, respectively, 
where we put $\mcO_X(1) \coloneqq \mcO_{\bP^{n+1}}(1)|_X$, 
and $\ST_{\mcO_X}$ denotes the spherical twist 
around the structure sheaf. 
Note that by taking the cohomology group of $X$, 
and restricting it to the subring 
$\Lambda_H(X) \subset H^*(X, \bQ)$ 
generated by the hyperplane class, 
the homomorphism (\ref{eq:introLS}) induces 
a usual local system $L$ with the fiber $\Lambda_H(X)$: 
\begin{equation} \label{eq:introloc}
    L \colon 
    \pi_1(\bP^1\setminus \{0, 1, \infty\}) \to \Aut(\Lambda_H(X)). 
\end{equation}
There is a canonical way to extend $L$ 
as a perverse sheaf on $\bP^1$, 
called the {\it intersection complex} and denoted by $\IC(L)$. 

The aim of this paper is to give a categorification of 
the intersection complex $\IC(L)$: 

\begin{thm}[Theorems \ref{thm:schober1} and \ref{thm:decatIC}] 
\label{thm:introschober1}
There exists a perverse schober $\fP$ on $\bP^1$ 
which extends the local system (\ref{eq:introLS}) of categories. 
More precisely, 
the perverse schober $\fP$ is a categorification of 
the intersection complex $\IC(L)$ associated to 
the local system (\ref{eq:introloc}). 
\end{thm}

A computation shows that 
the intersection complex $\IC(L)$ has 
a (anti-) symmetric pairing. 
In particular, it is Verdier self-dual. 
In Proposition \ref{prop:CYschober}, 
we will also prove that our perverse schober $\fP$ has 
a categorification of this property introduced in \cite{kss20}, 
see Definition \ref{def:CYschober} and 
Remark \ref{rmk:Vdual}. 

In the case of elliptic curves, 
we also construct a perverse schober $\fP^A$ on $\bP^1$ 
from the $A$-side and prove: 
\begin{thm}[Theorem \ref{thm:schoberA}] 
\label{thm:intromirror}
Let $X$ be an elliptic curve, $X^\vee$ its mirror. 
Then the perverse schober $\fP^A$ on $\bP^1$ 
has a generic fiber $D^\pi\Fuk(X^\vee)$, and 
it is identified with 
the perverse schober $\fP$ in Theorem \ref{thm:introschober1} 
under the mirror equivalence $D^b(X) \simeq D^\pi\Fuk(X^\vee)$. 
\end{thm}

\begin{rmk}
Our definition of perverse schobers is slightly more general than the original definition given in \cite{ks15}, 
see Definition \ref{def:schoberdisk} and Example \ref{ex:twodegat}. 
This generalization is essential for our construction. 
\end{rmk}

\subsection{Idea of proof}
To construct the perverse schober 
as in Theorem \ref{thm:introschober1}, 
we need to find spherical functors 
which induce the autoequivalences 
$(-)\otimes \mcO_X(1), \ST_{\mcO_X}$, and 
$\ST_{\mcO_X} \circ (\otimes \mcO_X(1))$. 
For the first two autoequivalences, 
we can find natural spherical functors 
using the derived categories of varieties. 

For the last equivalence 
$\ST_{\mcO_X} \circ (\otimes \mcO_X(1))$, 
we use the following Orlov equivalence

\[
D^b(X) \simeq \HMF(W), 
\]
to find the natural spherical functor,
where $W$ is the homogeneous polynomial 
defining $X \subset \bP^{n+1}$. 
We find that the autoequivalence 
$\ST_{\mcO_X} \circ (\otimes \mcO_X(1))$ 
has a natural presentation as the twist of a spherical functor 
between categories of graded matrix factorizations. 

Note that, for a given autoequivalence, 
there are many ways to express 
it as a twist of a spherical functor. 
We choose the natural expressions 
so that we recover the intersection complex 
after decategorification.

\subsection{Related works}
Various examples of perverse schobers, 
consisting of derived categories of varieties, 
have been constructed using birational geometry 
\cite{bks18, don19a, don19b, svdb19, svdb20}. 
Our perverse schober has a different origin from these examples, 
in the sense that it does not involve 
birational geometric structure. 
Moreover, it is the first example 
involving a Landau-Ginzburg model, 
which defines the category of graded matrix factorizations. 

Donovan--Kuwagaki \cite{dk21} verified 
a mirror symmetry type statement for some examples of 
perverse schobers coming from non-compact geometry, 
including the case of Atiyah flops. 
Theorem \ref{thm:intromirror} is an analogue of their result, 
for the case of elliptic curves.

\subsection{Open questions}
\begin{enumerate}
    \item It would be interesting to construct 
    a perverse schober using Fukaya categories, 
    which is mirror to our schober in higher dimension. 
    For example, in the case of a quartic K3 surface, 
    we need to understand mirror symmetry for the categories 
    $D^b(C)$ and $\HMF(\bC^3, W)$, 
    where $C$ is a smooth projective curve of genus $3$, 
    and $W$ is a homogeneous polynomial of degree $4$. 
    
    \item In the case of a quartic K3 surface, 
    we can think of the Riemann sphere $\bP^1$ as 
    a compactification of a certain quotient of 
    the space of normalized Bridgeland stability conditions. 
    It would be interesting to generalize our construction 
    to arbitrary K3 surfaces. 
\end{enumerate}

\subsection{Plan of the paper}
The paper is organized as follows: 
In Section \ref{sec:dgcat}, 
we review the theory of dg categories used in this paper. 
In particular, we recall the notion of spherical dg functors. 
In Section \ref{sec:def-schober}, we define the notion of 
perverse schobers on a Riemann surface following 
Kapranov--Schechtman \cite{ks15}. 
In Section \ref{sec:factorization}, 
we review the theory of derived factorization categories. 
In particular, we recall various constructions of 
dg enhancements of these categories. 

In Section \ref{sec:construct}, we construct our perverse schober. 
In Section \ref{sec:decat}, we prove that 
our perverse schober categorifies 
the intersection complex. 
In Section \ref{sec:mirror}, we discuss the mirror symmetry of 
perverse schobers for elliptic curves. 
Finally in Section \ref{sec:vgit}, 
we review the proof of Orlov equivalences 
via a variation of GIT quotients, following \cite{bfk19}. 
Using this, we construct an example of a spherical pair, 
which is another categorification of a perverse sheaf 
on a disk.

\begin{ACK}
The authors would like to thank Professor Arend Bayer, 
Dr.~Yuki Hirano, and Dr.~Kohei Kikuta for valuable discussions, 
and Professor Will Donovan for insightful comments on 
the previous version of this article. 
The authors would also like to thank the participants of 
the seminar on perverse schobers, 
held at the University of Edinburgh on Fall 2021, 
where they learned a lot about perverse schobers. 
In particular, they would like to thank Professor Pavel Safronov 
for organizing the seminar. 

N.K. was supported by 
ERC Consolidator grant WallCrossAG, no.~819864. 
G.O. is supported by JSPS KAKENHI Grant Number 19K14520. 

Finally, the authors would like to thank the referee for careful reading of the previous version of this paper and giving them a lot of useful comments. 
\end{ACK}

\begin{NaC}
Throughout the paper, we work over the complex number field $\bC$. 
\begin{itemize}
    \item For a variety $X$, $D^b(X)$ denotes 
    the bounded derived category of coherent sheaves on $X$. 
    \item For an integer $d \in \bZ$, 
    $\chi_d \colon \bC^* \to \bC^*$ denotes 
    the character defined by 
    $\chi_d(t) \coloneqq t^d$. 
\end{itemize}
\end{NaC}

\section{Quick review on dg categories} \label{sec:dgcat}
In this section, we briefly review the theory of dg categories 
used in this paper. 
We refer \cite{al17, kel94, kel06, toe07, toe11} 
for more details.

\subsection{Basic definitions}
We denote by $\dgcat_{\bC}$ 
the category of small dg categories over $\bC$. 
For a dg category $\mcA \in \dgcat_\bC$, 
we denote by $[\mcA]$ the {\it homotopy category} of $\mcA$. 

\begin{defin}
Let $\mcA, \mcB \in \dgcat_\bC$ be dg categories, 
and let $F \colon \mcA \to \mcB$ be a dg functor. 
We say that the functor $F$ is a {\it quasi-equivalence} 
if it satisfies the following two conditions: 
\begin{enumerate}
    \item For any objects $a, a' \in \mcA$, 
    the morphism 
    \[
    \Hom_\mcA(a, a') \to \Hom_\mcB(F(a), F(a')) 
    \]
    is a quasi-isomorphism. 
    
    \item The induced functor 
    $[F] \colon [\mcA] \to [\mcB]$ 
    on the homotopy categories is essentially surjective. 
\end{enumerate}

We denote by $\hodg_\bC$ the localization of $\dgcat_\bC$ 
by quasi-equivalences. 
A morphism in $\hodg_\bC$ is called a {\it quasi-functor}.  
\end{defin}

\subsubsection{Dg modules and derived categories}

We denote by $\Mod_\bC$ 
the dg category of complexes of $\bC$-vector spaces. 
Let $\mcA \in \dgcat_{\bC}$ be a dg category. 
A {\it right $\mcA$-module} is a dg functor 
$\mcA \to \Mod_\bC$. 
We denote by $\Mod_\mcA$
the dg category of right $\mcA$-modules. 
We have the {\it dg Yoneda embedding} 
\begin{equation} \label{eq:dgYoneda}
\mcA \hookrightarrow \Mod_\mcA, \quad 
a \mapsto \Hom_\mcA(-, a). 
\end{equation}

\begin{defin}
\begin{enumerate}
\item An object $C \in \Mod_\mcA$ is {\it acyclic} if 
for every $a \in \mcA$, 
the complex $C(a) \in \Mod_\bC$ is acyclic. 

\item An object $P \in \Mod_\mcA$ is {\it projective} if 
for every acyclic module $C \in \Mod_\mcA$, 
we have $\Hom_{\Mod_\mcA}(P, C)=0$. 
\end{enumerate}
\end{defin}

We denote by $\mcP(\mcA) \subset \Mod_\mcA$ 
the dg subcategory consisting of projective modules. 
The {\it derived category} $D(\mcA)$ is 
the localization of the homotopy category $[\Mod_\mcA]$ 
by acyclic $\mcA$-modules. 
The derived category $D(\mcA)$ has 
the structure of a triangulated category, 
and we have a canonical equivalence 
$[\mcP(\mcA)] \simeq D(\mcA)$. 

We denote by $\perf(\mcA) \subset D(\mcA)$ 
the full triangulated subcategory consisting of compact objects. 
A right $\mcA$-module is called {\it perfect} if 
its class in the derived category $D(\mcA)$ is compact. 
We denote by $\mcP^{\perf}(\mcA) \subset \mcP(\mcA)$ 
the dg subcategory of perfect projective modules. 
Note that we have an equivalence 
$[\mcP^{\perf}(\mcA)] \simeq \perf(\mcA)$. 

The dg Yoneda embedding (\ref{eq:dgYoneda}) 
induces the embedding 
\[
[\mcA] \hookrightarrow D(\mcA). 
\]
Note that the homotopy category $[\mcA]$ 
is not triangulated in general. 
We define the triangulated category 
$\tri(\mcA) \subset D(\mcA)$ to be 
the smallest triangulated subcategory containing $[\mcA]$. 
Then we have the following inclusions: 
\[
[\mcA] \subset \tri(\mcA) \subset \perf(\mcA) \subset D(\mcA). 
\]

\begin{defin}
A dg category $\mcA \in \dgcat_\bC$ is 
{\it pre-triangulated} (resp. {\it triangulated}) if 
the inclusion $[\mcA] \subset \tri(\mcA)$ 
(resp. $[\mcA] \subset \perf(\mcA)$) is an equivalence.
\end{defin}

\begin{rmk}
A pre-triangulated dg category $\mcA$ is triangulated 
if and only if its homotopy category $[\mcA]$ 
is idempotent complete 
(cf. \cite[Theorem 3.8]{kel06}). 
\end{rmk}

\subsubsection{Bimodules}
Let $\mcA, \mcB \in \dgcat_\bC$ be dg categories. 
An {\it $\mcA$-$\mcB$-bimodule} is 
an $\mcA^{\op} \otimes \mcB$-module. 
We denote by $_\mcA \Mod_\mcB$ the dg category of 
$\mcA$-$\mcB$-bimodules, 
and by $D(\mcA \mathchar`- \mcB)$ its derived category. 

\begin{defin}
Let $M$ be an $\mcA$-$\mcB$-bimodule. 
We say that $M$ is {\it $\mcA$-perfect} 
(resp. {\it $\mcB$-perfect}) if 
$M(b) \in \Mod \mathchar`- \mcA^{\op}$ 
(resp. $M(a) \in \Mod \mathchar`- \mcB$)
is perfect for all $b \in \mcB$ (resp. $a \in \mcA$). 

We denote by 
$D^{\mcA \mathchar`- \perf}(\mcA \mathchar`- \mcB)$ 
(resp. $D^{\mcB \mathchar`- \perf}(\mcA \mathchar`- \mcB)$) 
the full triangulated subcategory of 
$D(\mcA \mathchar`- \mcB)$ consisting of 
$\mcA$-perfect (resp. $\mcB$-perfect) bimodules. 
\end{defin}

Given a bimodule $M \in {_\mcA \Mod_\mcB}$, 
we have the {\it tensor product} functor 
\begin{align*}
    &(-) \otimes_\mcA M \colon \Mod_\mcA \to \Mod_\mcB, \quad 
    a \mapsto (b \mapsto b \otimes_\mcA M(a)), 
\end{align*}
and its derived functor 
\[
(-) \otimes^\dL_\mcA M \colon D(\mcA) \to D(\mcB). 
\]
Similarly, we have the functors 
\[
M \otimes_\mcB (-) \colon 
\Mod_{\mcB^{\op}} \to \Mod_{\mcA^{\op}}, \quad 
M \otimes^\dL_\mcB (-) \colon 
D(\mcB^{\op}) \to D(\mcA^{\op}). 
\]

We have the following characterizations of 
$\mcA$-perfect and $\mcB$-perfect bimodules
(see the first and second paragraphs in \cite[p2590]{al17}): 
\begin{itemize}
    \item $M$ is $\mcA$-perfect if and only if 
    the derived tensor product $(-) \otimes^\dL_\mcA M$ restricts 
    to the functor $\perf(\mcA) \to \perf(\mcB)$. 
    \item $M$ is $\mcB$-perfect if and only if 
    the derived tensor product $M \otimes^\dL_\mcB (-)$ restricts 
    to the functor $\perf(\mcB^{\op}) \to \perf(\mcA^{\op})$. 
\end{itemize}

\subsection{Dg enhancements}

We first define the notion of {\it dg enhancements}: 

\begin{defin}
Let $\mcD$ be a triangulated category. 
\begin{enumerate}
\item A {\it dg enhancement} of $\mcD$ is 
a pair $(\mcA, \epsilon)$ consisting of 
a pre-triangulated dg category $\mcA$ and an exact equivalence 
$\epsilon \colon [\mcA] \xrightarrow{\sim} \mcD$. 

\item A {\it Morita enhancement} of $\mcD$ is 
a pair $(\mcA, \eta)$ consisting of 
a dg category $\mcA$ and an exact equivalence 
$\eta \colon \perf(\mcA) \xrightarrow{\sim} \mcD$. 
\end{enumerate}
\end{defin}

\begin{rmk}
Suppose that a dg category $\mcA$ is triangulated. 
If we have a dg enhancement 
$(\mcA, \epsilon)$ of a triangulated category $\mcD$, 
it also gives a Morita enhancement via 
\[
\perf(\mcA) \simeq [\mcA] \xrightarrow{\epsilon} \mcD.  
\]
\end{rmk}

\begin{ex}
Let $\mcA \in \dgcat_\bC$ be a dg category. 
The dg category $\mcP^{\perf}(\mcA)$ 
is triangulated and 
gives a Morita enhancement of $\mcA$. 
\end{ex}

For our purpose, 
we also need the notion of enhancements 
of exact functors of triangulated categories.

\begin{thm}[\cite{toe07}] \label{thm:RHom}
Let $\mcA, \mcB \in \dgcat_\bC$ be dg categories. 
There exists a dg category 
$\dR\mcH om(\mcA, \mcB)$ with 
the following properties: 
\begin{enumerate}
    \item There exists a bijection 
    \[
    \Isom\left([\dR\mcH om(\mcA, \mcB)] \right) 
    = \Hom_{\hodg_\bC}(\mcA, \mcB), 
    \]
    where the left hand side denotes 
    the set of isomorphism classes of objects in 
    $[\dR\mcH om(\mcA, \mcB)]$. 
    
    \item There exists an equivalence 
    \[
    [\dR\mcH om(\mcP^{\perf}(\mcA), \mcP^{\perf}(\mcB))] 
    \simeq D^{\mcB \mathchar`- \perf}(\mcA \mathchar`- \mcB) 
    \]
    such that for an object 
    $M \in D^{\mcB \mathchar`- \perf}(\mcA \mathchar`- \mcB)$, 
    the corresponding exact functor 
    $[M] \colon \perf(\mcA) \to \perf(\mcB)$ 
    is $\otimes^\dL_\mcA M$. 
\end{enumerate}
\end{thm}
\begin{proof}
The existence of the dg category $\dR \mcH om (\mcA, \mcB)$ 
is proved in \cite[Theorem 6.1]{toe07}. 
The first assertion is \cite[Corollary 4.8]{toe07}; 
the second assertion is proved in \cite[Theorem 7.2]{toe07}. 
\end{proof}

\begin{defin}
Let $\Phi \colon \mcC \to \mcD$ 
be an exact functor of triangulated categories. 
Suppose that the triangulated categories $\mcC, \mcD$ have 
dg enhancements $\mcA, \mcB$, respectively. 
Then a {\it dg enhancement} of the functor $\Phi$ is 
an object $F \in [\dR\mcH om(\mcA, \mcB)]$ 
together with an isomorphism 
$[F] \simeq \Phi \colon \mcC \to \mcD$. 
\end{defin}

\subsection{Spherical functors}
Let $\mcA, \mcB \in \dgcat_\bC$ be dg categories, 
let $S \in D(\mcA \mathchar`- \mcB)$ be 
an $\mcA$-perfect and $\mcB$-perfect bimodule. 
Recall from Theorem \ref{thm:RHom} (2) that 
$S$ defines an isomorphism class of quasi-functors 
$\mcP^{\perf}(\mcA) \to \mcP^{\perf}(\mcB)$ 
whose underlying exact functor 
$\perf(\mcA) \to \perf(\mcB)$ is isomorphic to 
the derived tensor product 
$(-) \otimes^\dL_\mcA S$. 
We denote as 
\begin{equation} \label{eq:Stensor}
s \coloneqq (-) \otimes_\mcA^\dL S \colon D(\mcA) \to D(\mcB). 
\end{equation}
By \cite[Corollary 2.2]{al17}, there exist 
$\mcB$-perfect and $\mcA$-perfect objects 
$L, R \in D(\mcB \mathchar`- \mcA)$ 
such that the functors 
\[
l \coloneqq (-) \otimes^\dL_\mcB L, \quad 
r \coloneqq (-) \otimes^\dL_\mcB R \colon 
D(\mcB) \to D(\mcA) 
\]
are the left, right adjoints of the functor (\ref{eq:Stensor}), 
respectively. 

Let us denote by $SR$ the object 
$R \otimes^\dL_\mcA S \in D(\mcB \mathchar`- \mcB)$. 
We define objects $SL \in D(\mcB \mathchar`- \mcB)$, 
$RS, LS \in D(\mcA \mathchar`- \mcA)$ in a similar way. 
Then the corresponding derived tensor products are isomorphic to 
the functors $sr, sl, rs, ls$, respectively. 
By \cite[Definitions 2.3, 2.4]{al17}, there exist morphisms 
\begin{align*}
&SR \to \mcB , \quad \mcB \to SL, \quad 
\mcA \to RS, \quad LS \to \mcA,
\end{align*}
which induce adjoint (co)units on the underlying exact functors. 
Note that we regard $\mcA, \mcB$ as {\it diagonal} 
$\mcA$-$\mcA$-bimodule, $\mcB$-$\mcB$-bimodule, respectively. 
Namely, we define $\mcA \in {_\mcA \Mod_\mcA}$ as 
\[
\mcA \colon \mcA^{\op} \otimes \mcA \to \Mod_\bC, \quad 
(a, b) \mapsto \Hom_\mcA(b, a), 
\]
and similarly for $\mcB \in {_\mcB \Mod_\mcB}$. 
By taking (shifts of) cones, we obtain the exact triangles 
\begin{align*}
    &SR \to \mcB \to T, \\
    &T' \to \mcB \to SL, \\
    &F \to \mcA \to RS, \\
    &LS \to \mcA \to F'. 
\end{align*}
We call $T$ (resp. $T', F, F'$) as 
{\it twist} (resp. {\it dual twist, cotwist, dual cotwist}) of $S$. 
We denote by $t, t', f, f'$ their underlying exact functors. 

The following is the main result of \cite{al17}:
\begin{def-thm}[{\cite[Theorem 5.1]{al17}}]
Suppose that any two of the following conditions hold: 
\begin{enumerate}
    \item $t$ is an autoequivalence of $D(\mcB)$. 
    \item $f$ is an autoequivalence of $D(\mcA)$. 
    \item The composition $lt[-1] \to lsr \to r$ 
    is an isomorphism of functors. 
    \item The composition $r \to rsl \to fl[1]$ 
    is an isomorphism of functors. 
\end{enumerate}
Then all four hold. 
If this is the case, we call the object 
$S \in D(\mcA \mathchar`- \mcB)$ {\it spherical}.  
\end{def-thm}

The following result will be useful for our purpose: 
\begin{thm}[{\cite[Theorem B]{bar20}}] \label{thm:compsph}
Let $\mcA, \mcB, \mcC \in \dgcat_\bC$ be small dg categories. 
Suppose that we have spherical functors 
$D(\mcA) \to D(\mcC)$ and $D(\mcB) \to D(\mcC)$ whose twists 
are $t_\mcA$ and $t_\mcB$, respectively. 

Then there exists a dg category $\mcR \in \dgcat_\bC$ 
with a semi-orthogonal decomposition 
$D(\mcR)=\langle D(\mcA), D(\mcB) \rangle$, 
and a spherical functor $D(\mcR) \to D(\mcC)$ 
whose twist is $t_\mcA \circ t_\mcB$. 
\end{thm}

\section{Local systems of categories and perverse schobers} 
\label{sec:def-schober}
In this section, we recall the notion of perverse schobers 
on Riemann surfaces, 
which categorifies perverse sheaves. 
We refer \cite{don19a, don19b, ks15} for the details.

\subsection{Local systems of categories}

\begin{defin} \label{def:Loc}
Let $M$ be a manifold and 
let $Z \subset M$ be a finite subset. 
We define a {\it $Z$-coordinatized local system of categories} 
to be an action of the fundamental groupoid $\pi_1(M, Z)$ 
i.e., it consists of the following datum. 
\begin{itemize}
    \item A category $\mcD_z$ for each $z \in Z$, 
    \item An equivalence 
    $\rho_g \colon \mcD_z \to \mcD_{z'}$ 
    for each path $g \in \pi_1(M, Z)$, 
    \item A natural isomorphism 
    $\theta_{g, h} \colon \rho_g \rho_h \to \rho_{gh}$ 
    for each pair $(g, h)$ of composable paths, 
\end{itemize}
such that the diagrams 
\begin{equation} \label{eq:Loccomm}
\xymatrix{
&\rho_g \rho_h \rho_k 
\ar[r]^{\rho_g \theta_{h, k}} 
\ar[d]_{\theta_{g, h} \rho_k} 
&\rho_g \rho_{hk} \ar[d]^{\theta_{g, hk}} \\
&\rho_{gh} \rho_{k} \ar[r]^{\theta_{gh, k}} 
&\rho_{ghk}
}
\end{equation}
commute for all composable paths $g, h, k \in \pi_1(M, Z)$. 
\end{defin}

\subsection{Quiver description of perverse sheaves on a disk}
To motivate the definition of perverse schobers, 
we first recall the quiver description of the category of 
usual perverse sheaves on a disk. 

Let $\Delta$ be the unit disk, 
$B=\{b_1, \cdots, b_n \} \subset \Delta$ a finite subset. 
We denote by $\Perv(\Delta, B)$ the category of 
perverse sheaves on $\Delta$ singular at $B$ 
(i.e., whose restrictions to $\Delta \setminus B$ 
are local systems). 

\begin{defin} \label{def:catPn}
We define $\mcP_n$ to be the category of data 
$(D, D_i, u_i, v_i)_{i=1}^n$ 
consisting of finite dimensional vector spaces $D, D_i$ 
and linear maps $u_i \colon D \to D_i$, $v_i \colon D_i \to D$ 
such that the endomorphisms 
$T_i:=\id - v_i \circ u_i \colon D \to D$ 
are isomorphisms for all $i = 1, \cdots, n$. 
\end{defin}

Fix a point $p \in \partial \Delta$. 
A {\it skeleton} is a union of simple arcs joining $p$ and $b_i$ 
for all $b_i \in B$, coninciding near $p$. 
Let $\mcC$ be the set of isotopy classes of skeletons. 
We will see that 
there are equivalences between the categories 
$\Perv(\Delta, B)$ and $\mcP_n$ 
parametrized by the set $\mcC$, 
compatible with the natural action of 
the Artin braid group 
\[
\Br_n:=\left\langle 
    s_1, \cdots, s_{n-1}: 
    s_is_{i+1}s_i = s_{i+1}s_is_{i+1}
    \right\rangle. 
\]

For a generator $s_i \in \Br_n$, 
we define an autoequivalence 
$f_{s_i} \colon \mcP_n \to \mcP_n$ 
as follows: 
For an object $(D, D_j, u_j, v_j)_{j=1}^n \in \mcP_n$, 
we define 
\[
f_{s_i}(D, D_j, u_j, v_j):= (D', D_j', u_j', v_j'), 
\]
where 
\begin{equation} \label{eq:braid}
\begin{aligned}
    &D'=D, \quad D_j'=D_j, \quad u_j'=u_j, \quad v_j'=v_j 
    \quad (j \neq i, i+1), \\
    &D_{i+1}'=D_i, \quad D_i'=D_{i+1}, \\
    &u_i'=u_{i+1}, \quad v_i'=v_{i+1}, \quad 
    u_{i+1}'=u_iT_{i+1}, \quad v_{i+1}'=T_{i+1}^{-1}v_i. 
\end{aligned}
\end{equation}
For a general element $\sigma \in \Br_n$, 
we then define an equivalence 
$f_\sigma \colon \mcP_n \to \mcP_n$ 
as the composition of $f_{s_i}$'s and their inverses. 

\begin{prop}[\cite{gmv96}] \label{prop:quiver}
For each class $K \in \mcC$, 
there is an equivalence 
\[
F_K \colon \Perv(\Delta, B) \to \mcP_n 
\]
of categories, such that, 
for every $\sigma \in \Br_n$, 
the following diagram commutes: 
\[
\xymatrix{
&\mcP_n \ar[rr]^{f_\sigma} & &\mcP_n \\
& &\Perv(\Delta, B). \ar[ul]^{F_K} \ar[ur]_{F_{\sigma(K)}} &
}
\]
\end{prop}

\subsection{Perverse schobers on Riemann surfaces}
First we define the notion of perverse schobers on a disk. 
We keep the notations in the previous subsection. 

\begin{defin}[\cite{ks15, don19b}] \label{def:schoberdisk}
\begin{enumerate}
    \item A {\it coodinatized schober} on $(\Delta, B)$ 
    is a datum 
    $\fS=(\mcD, \mcD_i, S_i, \epsilon_i)_{i=1}^n$ 
    consisting of 
    triangulated dg categories $\mcD, \mcD_i$, 
    spherical functors $S_i \colon \mcD_i \to \mcD$, 
    and numbers $\epsilon_i \in \{\pm 1\}$. 
    For each $i$, we denote by $R_i, L_i$ 
    the right and left adjoints of $S_i$, 
    and by $T_i:=\Cone(S_iR_i \to \id_\mcD)$ 
    (resp. $T'_i \coloneqq \Cone({id_\mcD} \to S_iL_i)[-1]$)
    the corresponding twist (resp. dual twist) equivalence. 
    \item For a given coordinatized schober $\fS$ 
    and an element $\sigma \in \Br_n$, 
    we define a new coordinatized schober $f_\sigma(\fS)$ 
    as in (\ref{eq:braid}), replacing 
    $(D, D_i, u_i, v_i)$ with $(\mcD, \mcD_i, R_i, S_i)$ 
    (resp. $(\mcD, \mcD_i, L_i, S_i)$), 
    and endomorphism $T_i=\id-v_iu_i$ 
    with the functor $T_i$ 
    (resp. $T'_i$) 
    when $\epsilon_i=1$ (resp. $\epsilon_i=-1$). 
    \item A {\it perverse schober} on $(\Delta, B)$ is 
    a collection $(\fS_K)_{K \in \mcC}$ of 
    coordinatized schobers $\mathfrak{S}_K$ 
    such that for each $\sigma \in \Br_n$, 
    we have a compatible identification 
    \[
    f_{\sigma} \fS_K \xrightarrow{\sim} \fS_{\sigma(K)}. 
    \]
\end{enumerate}
\end{defin}

\begin{rmk} \label{rmk:dual=inverse}
By \cite[Proposition 5.3]{al17}, 
we have an isomorphism $T'_i \cong T_i^{-1}$ 
for each $i$. 
\end{rmk}

A choice of signs $\epsilon_i$ in the above definition 
is important when we consider decategorifications: 
\begin{ex} \label{ex:twodegat}
Let $X$ be a smooth projective Calabi--Yau variety of dimension $n$. Let $E \in D^b(X)$ be a spherical object, 
e.g., $E=\mcO_X$. 
Then we obtain a spherical functor 
\[
S \colon D^b(\pt) \to D^b(X), \quad V \mapsto V \otimes E. 
\]
Denote by $L, R$ its left and right adjoints. 

Taking the cohomology groups, we get the following diagrams: 
\begin{equation} \label{eq:twodecat}
\xymatrix{
&H^*(\pt) \ar@<1ex>[r]^-{S^H} &H^*(X) \ar@<1ex>[l]^-{R^H}, 
&H^*(\pt) \ar@<1ex>[r]^-{S^H} &H^*(X) \ar@<1ex>[l]^-{L^H}, 
}
\end{equation}
where $S^H, L^H, R^H$ denote the cohomological Fourier--Mukai transforms. 
Since the canonical bundles of $\pt$ and $X$ are trivial, 
we have $L^H=(-1)^{n}R^H$. 
Hence if $n$ is odd, the diagrams (\ref{eq:twodecat}) define two different perverse sheaves on a disk. 
\end{ex}

\begin{defin}
Let $\fS=(\mcD, \mcD_i, S_i, \epsilon_i)_{i=1}^n$ 
be a coordinatized perverse schober on $(\Delta, B)$. 
We define an {\it induced local system} of categories 
on $\partial\Delta$ to be the one whose monodromy 
autoequivalence is 
\[
T_1^{\epsilon_1} \circ \cdots T_n^{\epsilon_n}. 
\]
\end{defin}

We now define the notion of perverse schobers on Riemann surfaces. 
Let $\Sigma$ be a Riemann surface and 
$B \subset \Sigma$ a finite set of points. 
Take a disk $\Delta \subset \Sigma$ 
containing the set $B$. 
Denote by $U$ the closure of $\Sigma \setminus \Delta$. 

We fix a point $p \in \partial \Delta$ 
and a finite set $Y \subset \Sigma \setminus B$. 
We set $Z \coloneqq Y \cup \{p\}$. 

\begin{defin}[\cite{ks15, don19b}] \label{def:schober}
A {\it perverse schober} on $(\Sigma, B)$ is a datum 
consisting of 
\begin{itemize}
    \item a perverse schober $\fS_{\Delta}$ on $(\Delta, B)$, 
    \item a $Z$-coordinatized local system $\fL$ 
    of categories on $U$, 
\end{itemize}
such that the induced $\{p\}$-coordinatized local systems 
$\fS_\Delta |_{\partial \Delta}$ and 
$\fL|_{\partial \Delta}$ are isomorphic. 
\end{defin}

\begin{ex} \label{ex:trivialmonod}
Consider the case $\Sigma=\bP^1$. 
In this case, giving a perverse schober on $(\bP^1, B)$ is 
equivalent to giving a perverse schober on $(\Delta, B)$ 
such that the induced local system on $\partial\Delta$ has 
the trivial monodromy. 
\end{ex}

The following result is very useful for our purpose: 
\begin{prop}[{\cite[Proposition 4.10]{don19b}}]
\label{prop:extension}
Fix a point $p \in \Sigma \setminus B$. 
Let $\fS_{\Sigma \setminus \{p\}}$ be a schober 
on $\left(\Sigma \setminus \{p\}, B \right)$, 
$T \colon \mcD \to \mcD$ 
be the corresponding monodromy autoequivalence 
around the point $p$. 

Given a presentation of the equivalence $T$ 
as the twist or the dual twist of a spherical functor, 
we obtain a schober on $\left(\Sigma, B \cup \{p\} \right)$.
\end{prop}

\begin{rmk}
Note that there are many ways presenting an equivalence $T$ 
as a spherical (dual) twist, 
and a schober on $\Sigma$ in the above proposition 
depends on such a choice. 
We will find the most natural presentations in our setting. 
\end{rmk}

Finally we recall the notion of {Calabi-Yau} perverse schobers: 
\begin{defin}[{\cite[Definitions 3.1.8 and 3.1.14]{kss20}}] 
\label{def:CYschober}
Let $n \in \bZ_{>0}$ be a positive integer. 
\begin{enumerate}
    \item A perverse schober on $(\Delta, 0)$, 
    represented by a spherical functor $\mcD_1 \to \mcD$, 
    is {\it $n$-Calabi-Yau (CY)} 
    if the following conditions hold: 
    \begin{enumerate}
        \item The category $\mcD$ is CY of dimension $n$, 
        \item The shifted cotwist $f[n+1]$ is 
        the Serre functor of $\mcD_1$. 
    \end{enumerate}
    
    \item A perverse schober $\fS$ on $(\Sigma, B)$ is 
    {\it $n$-Calabi-Yau} if the following conditions hold: 
    \begin{enumerate}
        \item The restriction $\fS|_{\Sigma \setminus B}$ is 
        a local system of $n$-CY categories, 
        and the monodromy autoequivalences preserve 
        the CY structure, 
        \item For each point $b \in B$, 
        the restriction of $\fS$ to a disk around $b$ is 
        $n$-CY. 
    \end{enumerate}
\end{enumerate}
\end{defin}

\begin{rmk} \label{rmk:Vdual}
The notion of $n$-CY perverse schobers categorifies 
perverse sheaves with (anti-)symmetric structures 
with respect to the Verdier duality, 
see \cite[Proposition 1.3.22]{kss20}. 
In particular, these perverse sheaves are Verdier self-dual. 
\end{rmk}

\subsection{Spherical pairs}
Consider the special case where 
$\Sigma=\Delta$ and $B=\{0\}$. 
In this case, we have another categorification of 
perverse sheaves on $(\Delta, 0)$, 
called {\it spherical pair}: 

\begin{defin}
\label{def:Spair}
A {\it spherical pair} is 
a pair of semi-orthogonal decompositions 
\[
\mcE_0
=\langle
    \mcE_-^{\perp}, \mcE_-
    \rangle
=\langle
    \mcE_+^{\perp}, \mcE_+
    \rangle
\]
of a triangulated category $\mcE_0$ 
such that the compositions of 
the inclusions and the projections 
\[
\mcE_{\pm}^\perp \to \mcE_{\mp}^\perp, \quad 
\mcE_\pm \to \mcE_\mp
\]
are equivalences. 
\end{defin}

\begin{rmk} 
As proved in \cite[Proposition 3.7]{ks15}, 
given a spherical pair, 
we obtain a spherical functor 
$\mcE_- \to \mcE_+^\perp$. 
Hence a spherical pair induces 
a perverse schober on $(\Delta, 0)$. 
\end{rmk}

As we will see in Section \ref{sec:vgit}, 
the theory of variations of GIT quotients, 
developed by \cite{bfk19, hl15}, 
provides natural classes of spherical pairs.

\section{Derived factorization category} \label{sec:factorization}
\subsection{Derived factorization category}
In this subsection, we recall the notion of  derived factorization categories following \cite{bfk14}.

Let $G$ be an affine algebraic group acting on a variety $X$, and $\chi:G \to \bC^*$ be a character of $G$. Let $W:X\rightarrow \bC$ be a $\chi$-semi-invariant regular function, i.e.  $W(g\cdot x)=\chi(g)W(x)$ for any $g\in G$ and any $x\in X$. The data $(X,\chi,W)^G$ is called a {\it gauged Landau-Ginzburg  model}. Denote the character invertible sheaf of $\chi$ on $X$ by $\mcO_X(\chi)$. First, we define the dg category $\fact_G(X,\chi,W)$ of factorizations of $(X,\chi,W)^G$. 

\begin{defin}
A {\it factorization}  of $(X,\chi,W)^G$ is a sequence
\begin{equation} \label{eq:deffact}
E=\Bigl(E_1\xrightarrow{\varphi_1^E} E_0\xrightarrow{\varphi_0^E} E_1(\chi)\Bigr)
\end{equation}
of morphisms of $G$-equivariant coherent sheaves on $X$ such that
\[\varphi_0^E\circ\varphi_1^E=W\cdot {\rm id}_{E_1}, \ \varphi_1^E(\chi)\circ\varphi_0^E=W\cdot {\rm id}_{E_0}.\] 
A factorization $E=\Bigl(E_1\xrightarrow{\varphi_1^E} E_0\xrightarrow{\varphi_0^E} E_1(\chi)\Bigr)$ of $(X,\chi,W)^G$ is called a  {\it locally free factorization} if $E_1$ and $E_0$ are $G$-equivariant locally free sheaves on $X$.
\end{defin}

\begin{defin}\label{def:Fact}
Let \[E=\Bigl(E_1\xrightarrow{\varphi_1^E} E_0\xrightarrow{\varphi_0^E} E_1(\chi)\Bigr), 
F=\Bigl(F_1\xrightarrow{\varphi_1^F} F_0\xrightarrow{\varphi_0^F} F_1(\chi)\Bigr)\]
be factorizations of $(X,\chi,W)^G$. We define the $\bZ$-graded $\bC$-vector space $\Hom_{\fact_G(X,\chi,W)}(E,F)$ of morphisms from $E$ to $F$ as follows. 
For an integer $l$, we define
\[\Hom^{2l}_{\fact_G(X,\chi,W)}(E,F):=\Hom(E_1,F_1(\chi^l)) \oplus \Hom(E_0,F_0(\chi^l)),\]
\[\Hom^{2l+1}_{\fact_G(X,\chi,W)}(E,F):=\Hom(E_1, F_0(\chi^l)) \oplus \Hom(E_0,F_1(\chi^{l+1})).\]
Then we define 
\[\Hom_{\fact_G(X,\chi,W)}(E,F):=\bigoplus_{n \in \bZ}\Hom^n_{\fact_G(X,\chi,W)}(E,F).\]

Moreover, we define the linear map
\[d_{E,F}:\Hom_{\fact_G(X,\chi,W)}(E,F) \to \Hom_{\fact_G(X,\chi,W)}(E,F)\]
as follows.
Take $f=(f_1,f_0) \in \Hom^{n}_{\fact_G(X,\chi,W)}(E,F)$.
When $n=2l$ for some integer $l$, we put 
\[d_{E,F}(f):=(\varphi^F_1(\chi^l) \circ f_1-f_0 \circ \varphi^E_1, \varphi^F_0(\chi^l) \circ f_0-f_1(\chi) \circ \varphi^E_0).\]
When $n=2l+1$ for some integer $l$, we put
\[d_{E,F}(f):=(\varphi^F_0(\chi^l) \circ f_1+f_0 \circ \varphi^E_1, \varphi_1^F(\chi^{l+1}) \circ f_0+f_1(\chi) \circ \varphi_0^E).\]
Then $d_{E,F}^2=0$ holds and the degree of $d_{E,F}$ is one. We have a dg $\bC$-module $(\Hom_{\fact_G(X,\chi,W)}(E,F), d_{E,F})$.
\end{defin}

Thus, we have the dg category of factorizations of $(X,\chi,W)^G$.

\begin{defin}\label{def:Fact2}
Denote the dg category of factorizations of $(X,\chi,W)^G$ by $\fact_G(X,\chi,W)$. 
Let $\vect_G(X,\chi,W)$ be the dg subcategory of locally free factorizations of $(X,\chi,W)^G$. 
\end{defin}

We obtain triangulated categories from the dg categories in Definition \ref{def:Fact2}.

\begin{prop}[{\cite[Proposition 3.7]{bfk14}}]\label{prop:Fact_tri}
The dg categories $\fact_G(X,\chi,W)$ and $vect_G(X,\chi,W)$ are triangulated dg categories.
\end{prop}
 
Next, we define the derived factorization categories $\Dabs\fact_G(X,\chi,W)$  and $\Dabs[\fact_G(X,\chi,W)]$ of $(X,\chi,W)^G$.

\begin{rmk}
The category $Z^0(\fact_G(X,\chi,W))$ has the structure of an abelian category,
where the category $Z^0(\fact_G(X,\chi,W))$ is defined as follow. The objects of $Z^0(\fact_G(X,\chi,W))$ are same as $\fact_G(X,\chi,W)$. For objects $E,F \in Z^0(\fact_G(X,\chi,W))$, we define $\Hom_{Z^0(\mcA)}(E,F):=\mathrm{Ker}d^0_{E,F}$.

\end{rmk}

To introduce the notion of acyclic objects, we need the totalization of an object in $Z^0(\fact_G(X,\chi,W))$.

\begin{defin}
Let $E^{\text{\tiny{\textbullet}}}=(\cdot\cdot\cdot\rightarrow E^i\xrightarrow{\delta^i}E^{i+1}\rightarrow\cdot\cdot\cdot)$ be a complex of objects in the abelian category $Z^0(\fact_G(X,\chi,W))$.
For $i \in \bZ$, we put
\[E^i=\Bigl(E^i_1\xrightarrow{\varphi_1^{E^i}} E^i_0\xrightarrow{\varphi_0^{E^i}} E^i_1(\chi)\Bigr).\]
We define $G$-equivariant quasi-coherent sheaves $T_1, T_0$ on $X$ as
\[T_1:=\bigoplus_{l \in \bZ}\Bigl(E^{-2l-1}_0(\chi^l) \oplus E^{-2l-2}_1(\chi^{l+1})\Bigr), \]
\[T_0:=\bigoplus_{l \in \bZ}\Bigl(E^{-2l}_0(\chi^l) \oplus E^{-2l-1}_1(\chi^{l+1}) \Bigr).\]
We define morphisms $t_1, t_0$ of $G$-equivariant quasi-coherent sheaves as 
\[t_1:=\sum_{l \in \bZ}\Bigl(\Bigl(\delta^{-2l-1}_0( \chi^l) - \varphi^{E^{-2l-1}}_0( \chi^l)\Bigr)\oplus \Bigl(\delta^{-2l-2}_1(\chi^{l+1})+\varphi^{E^{-2l-2}}_1(\chi^{l+1})\Bigr)\Bigr),\]
\[t_0:=\sum_{l \in \bZ}\Bigl(\Bigl(\delta^{-2l}_0( \chi^l)+ \varphi^{E^{-2l}}_0( \chi^l)\Bigr)\oplus \Bigl(\delta^{-2l-1}_1(\chi^{l+1})- \varphi^{E^{-2l-1}}_1(\chi^{l+1})\Bigr)\Bigr).\]
The {\it totalization} $\mathrm{Tot}(E^{\text{\tiny{\textbullet}}}) \in \fact_G(X,\chi,W)$ of $E^{\text{\tiny{\textbullet}}}$ is an object defined as
\[ \mathrm{Tot}(E^{\text{\tiny{\textbullet}}})=\Bigl(T_1\xrightarrow{t_1} T_0\xrightarrow{t_0} T_1(\chi)\Bigr) .\]
\end{defin}

\begin{defin}
A factorization $E \in \fact_G(X,\chi,W)$ is called {\it acyclic} if there is an acyclic complex $E^{\text{\tiny{\textbullet}}}$ of objects in $Z^0(\Fact_G(X,\chi,W))$ such that $E$ is isomorphic to $\mathrm{Tot}(E^{\text{\tiny{\textbullet}}})$ in $Z^0(\fact_G(X,\chi,W))$.
Let $\acyc_G(X,\chi,W)$ be the full dg subcategory of acyclic factorizations in $\fact_G(X,\chi,W)$. 
\end{defin}


\begin{defin}
We define the {\it absolute derived category}  $D^{\abs}[\fact_G(X,\chi,W)]$
of  $[\fact_G(X,\chi,W)]$ to be the Verdier quotient
\[D^{\abs}[\fact_G(X,\chi,W)]:=[\fact_G(X,\chi,W)]/[\acyc_G(X,\chi,W)].\]
\end{defin}

\begin{defin}
We define the {\it absolute derived category}  $D^{\abs}\fact_G(X,\chi,W)$ of $\fact_G(X,\chi,W)$
to be the dg quotient
\[D^{\abs}\fact_G(X,\chi,W):=\fact_G(X,\chi,W)/\acyc_G(X,\chi,W).\]
\end{defin}

The categories $\Dabs\fact_G(X,\chi,W)$ and $\Dabs[\fact_G(X,\chi,W)]$ are also called the {\it derived factorization categories} of $(X,\chi,W)^G$.

\begin{prop}[{\cite[Corollary 5.9, Proposition 5.11]{bfk14}}]\label{prop:quasi-eq}
The dg category $D^{\abs}\fact_G(X,\chi,W)$ is a dg enhancement of
$D^{\abs}[\fact_G(X,\chi,W)]$. If $X$ is smooth affine and $G$ is reductive, the canonical dg functor $\vect_G(X,\chi,W) \to D^{\abs}\fact_G(X,\chi,W)$ is a quasi-equivalence.
\end{prop}

\subsection{Derived functors}
In this subsection, we recall the construction of derived functors which we will use later.
Let $X$ and $Y$ be smooth affine varieties, and $G$ a reductive affine algebraic group acting on $X$ and $Y$. Let $f:X \to Y$ be a $G$-equivariant morphism. Take a character $\chi:G \to \bC^*$ of $G$. 
Let $W:Y \to \bC$ be a $\chi$-semi-invariant regular function on $Y$. 
Then $(Y,\chi,W)^G$ and $(X,\chi,f^*W)^G$ are gauged Landau-Ginzburg models.

\begin{defin}
The dg functor 
$f^*:\vect_G(Y,\chi,W) \to \vect_G(X,\chi,f^*W)$ 
is defined as follows. 
For an object $E \in \vect_G(Y,\chi,W)$, 
we define an object
\[f^*(E):=\Bigl(f^*E_1\xrightarrow{f^*\varphi_1^E} f^*E_0\xrightarrow{f^*\varphi_0^E} f^*E_1(\chi)\Bigr)\in \vect_G(X,\chi, f^*W).\]
For a morphism $p=(p_1,p_0) \in \Hom_{\vect_G(Y,\chi,W)}(E,F)$, we define a morphism 
\[f^*p:=(f^*p_1, f^*p_0) \in \Hom_{\vect_G(X,\chi,f^*W)}(f^*E,f^*F).\]
By Proposition \ref{prop:quasi-eq}, we have the exact functor
\[D^{\abs}[\fact_G(Y,\chi,W)]\to D^{\abs}[\fact_G(X,\chi,f^*W)]\]
and denote it by $f^*$. 
\end{defin}

\begin{defin}
Assume that $f$ is a closed immersion.
We define the dg functor $f_*:\fact_G(X,\chi,f^*W) \to \fact_G(Y,\chi,W)$ as follows. For an object $E \in \fact_G(X,\chi,f^*W)$, 
we define an object
\[f_*(E):=\Bigl(f_*E_1\xrightarrow{f_*\varphi_1^E} f_*E_0\xrightarrow{f_*\varphi_0^E} f_*E_1(\chi)\Bigr)\in \fact_G(Y,\chi, W).\]
For a morphism $p=(p_1,p_0) \in \Hom_{\fact_G(X,\chi,f^*W)}(E,F)$, we define a morphism 
\[f_*p:=(f_*p_1, f_*p_0) \in \Hom_{\fact_G(Y,\chi,W)}(f_*E,f_*F).\]
Since $f$ is a closed immersion, $f:\fact_G(X,\chi,f^*W) \to \fact_G(Y,\chi,W)$ sends acyclic factorizations to acyclic factorizations. 
Therefore, we obtain the dg functor 
$f_*:\Dabs\fact_G(X,\chi,f^*W) \to \Dabs\fact_G(Y,\chi,W)$ and the exact functor $f_*:\Dabs[\fact_G(X,\chi,f^*W)]\to \Dabs[\fact_G(Y,\chi,W)]$.
\end{defin}

\begin{defin}
Let $(X,\chi,W)^G$ be a gauged Landau-Ginzburg model.
Let $L$ be a $G$-equivariant line bundle on $X$. We define the dg functor $- \otimes L: \fact_G(X,\chi,W) \xrightarrow{\sim} \fact_G(X,\chi,W)$ as follows.
For an object $E \in \fact_G(X,\chi,W)$, we define an object 
\[E \otimes L:=\Bigl(E_1\otimes L \xrightarrow{\varphi^E_1\otimes \mathrm{id}_L}E_1\otimes L\xrightarrow{\varphi^E_0(\chi)\otimes \mathrm{id}_L}E_1 \otimes L(\chi)\Bigr).\]
For a morphism $f=(f_1,f_0) \in \Hom_{\fact_G(X,\chi,W)}(E,F)$, we define a morphism 
\[f \otimes L:=(f_1 \otimes \mathrm{id}_L, f_0 \otimes \mathrm{id}_L)\in \Hom_{\fact_G(X,\chi,W)}(E \otimes L, F \otimes L).\]
Restricting this dg functor to the full dg subcategory $\vect_G(X,\chi,W)$, we obtain the dg functor $-\otimes L:\vect_G(X,\chi,W) \xrightarrow{\sim} \vect_G(X,\chi,W)$. 
Note that it induces the autoequivalence 
\[- \otimes L:D^{\abs}[\fact_G(X,\chi,W)] \xrightarrow{\sim} D^{\abs}[\fact_G(X,\chi,W)].\]
\end{defin}

More generally, 
for $\mcE \in \fact_G(X,\chi,W)$, we have the derived tensor functor
\[ -\otimes^{\mathbf{L}}\mcE:\Dabs[\fact_G(X,\chi,0)] \to \Dabs[\fact_G(X,\chi,W)]\] by \cite[Proposition 4.23]{H17}.
By \cite[Definition 3.14]{H17}, there is the exact functor $\Upsilon:D^b(\Coh_G(X)) \to \Dabs[\fact_G(X,\chi,0)]$ sending a $G$-equivariant coherent sheaf $A$ to the factorization $(0 \to A \to 0)$. 
For simplicity, we also denote the composition $\Upsilon(-)\otimes^{\mathbf{L}}\mcE$ by $-\otimes^{\mathbf{L}}\mcE$ as in \cite[Definition 4.24]{H17}.

\subsection{Koszul factorizations}
Let $(X, \chi, W)^G$ be a gauged Landau-Ginzburg model.
Assume that $X$ is smooth. Let $\mcE$ be a $G$-equivariant locally free sheaf on $X$ of finite rank. Let $s:\mcE \to \mcO_X$ and $t:\mcO_X \to \mcE(\chi)$ be morphisms of $G$-equivariant coherent sheaves such that $t \circ s=W \cdot \id_{\mcE}$ and $s(\chi)\circ t=W \cdot \id_{\mcO_X}$. Let $Z_s \subset X$ be the zero scheme of the section $s \in H^0(\mcE^\vee)^G$.
The section $s$ is called regular if the codimension of $Z_s$ in $X$ is equal to the rank of $\mcE$.

\begin{defin}
We define the {\it Koszul factorization} 
$K(s,t) \in \vect_G(X,\chi,W)$ of $s$ and $t$ as 
\[K(s.t):=\Bigl(K_1\xrightarrow{\varphi_1^K} K_0\xrightarrow{\varphi_0^K} K_1(\chi)\Bigr),\]
where 
\[K_1:=\bigoplus_{n \geq 0}\bigl(\bigwedge^{2n+1}\mcE\bigr)\left(\chi^n\right), \quad 
K_0:=\bigoplus_{n \geq 0}\bigl(\bigwedge^{2n}\mcE\bigr)\left(\chi^n\right)\]
and \[\varphi^K_0,\varphi^K_1:=\vee s+t \wedge-.\]
\end{defin}

We will use the following lemma later: 

\begin{lem}[{\cite[Proposition 3.20, Lemma 3.21]{bfk14}}] \label{lem:Koszul}
The following statements hold: 
\begin{itemize}
    \item[(1)]There is the natural isomorphism 
    \[K(s,t)^\vee \simeq K(t^\vee, s^\vee).  \]
    
    \item[(2)]If $s$ is regular, there are the natural isomorphisms 
    \[\mcO_{Z_s} \simeq K(s,t), \ \mcO_{Z_s} \otimes \bigwedge^{\rk \mcE}\mcE[-\rk \mcE]\simeq K(s,t)^\vee \]
    in $\Dabs[\fact_G(X,\chi,W)]$
\end{itemize}
\end{lem}

\subsection{Kn\"{o}rrer periodicity}
Let $X$ be a smooth quasi-projective variety. Consider the trivial action of $\bC^*$ on $X$.
Let $\mcE$ be a $\bC^*$-equivariant locally free sheaf on $X$ of finite rank. 
Take a $\bC^*$-invariant regular section $s \in H^0(\mcE^\vee)^{\bC^*}$. Then we have the $\chi_1$-semi-invariant regular function $Q_s:V(\mcE(\chi_1)) \to \bC$ induced by $s$. 
Let $V(\mcE(\chi_1)):=\Spec_X(\Sym(\mcE(\chi_1)^\vee))$ be the $G$-equivariant vector bundle on $X$. Take the restriction $p:V(\mcE(\chi_1))|_{Z_s} \to Z_s$ of the $G$-equivariant vector bundle $q$ to $Z_s$. Then there is the following commutative  diagram:
\begin{equation*} 
\xymatrix{
&V(\mcE(\chi_1))|_{Z_s} \ar[r]^-i 
\ar[d]_-{p} 
&V(\mcE(\chi_1)) \ar[d]^-q \\
&Z_s \ar[r]_-j & X. 
}
\end{equation*}

Shipman \cite{S12} proved the following theorem. See also \cite{Isik}, \cite[Theorem 4.1]{H17}. 
\begin{thm}[{\cite[Theorem 3.4]{S12}}] \label{thm:Knorrer}
We have the equivalence
\begin{equation}\label{Knorrer periodicity}
 i_*\circ p^*:\Dabs[\fact_{\bC^*}(Z_s,\chi_1,0)] \xrightarrow{\sim} \Dabs[\fact_{\bC^*}(V(\mcE(\chi_1)), \chi_1, Q_s)]. 
\end{equation}
\end{thm}

The equivalence in Theorem \ref{thm:Knorrer} is called the {\it Kn\"orrer periodicity}.
By \cite[Proposition 2.14]{H17}, there is the canonical equivalence 
\begin{equation}\label{canonical equivalence}
    D^b(Z_s) \xrightarrow{\sim} \Dabs[\fact_{\bC^*}(Z_s,\chi_1,0)].
\end{equation} 
We also denote by $i_* \circ p^*$ the composition of the equivalences (\ref{Knorrer periodicity}) and (\ref{canonical equivalence}).

 \subsection{Graded matrix factorizations}

Let $S_n:=\bC[x_1,\cdots, x_n]$ be the polynomial ring. 
We regard $S_n$ as a graded $\bC$-algebra, where 
the degree of $x_i$ is one for each $i = 1, \cdots, n$. 
Let $\mathrm{gr}(S_n)$ be the category of finitely generated graded $S_n$-modules whoese morphisms are degree preserving homomorphisms. 
Let  $\mathrm{grfr}(S_n) \subset \mathrm{gr}(S_n)$ be the full subcategory of finitely generated graded free $S_n$-modules. 

We recall the definition of graded matrix factorizations.
Fix a positive integer $d \in \bZ_{>0}$. Let $W \in S_n$ be a homogeneous polynomial of degree $d$.

\begin{defin}
A {\it graded matrix factorization}  of $W$ is a sequence
\begin{equation} \label{eq:defgrfact}
M=\Bigl(M_1\xrightarrow{\varphi_1^M} M_0\xrightarrow{\varphi_0^M} M_1(d)\Bigr),
\end{equation}
 where $M_1$ and $M_0$ are finitely generated free graded $S_n$-modules and $ \varphi_1^M$ and $\varphi_0^M$ are degree preserving homomorphisms such that
\[\varphi_0^M\circ\varphi_1^M=W\cdot {\rm id}_{M_1}, \ \varphi_1^F(d)\circ\varphi_0^F=W\cdot {\rm id}_{M_0}.\] 
\end{defin}

As similar to Definitions \ref{def:Fact} and \ref{def:Fact2}, 
we can define the {\it dg category} $\MF(W)$ {\it of graded matrix factorizations of} $W$. Let $\HMF(W)$ be the homotopy category $[\MF(W)]$.
For an integer $l \in \bZ$ and a graded $S_n$-module $P=\oplus_{i \in \bZ} P_i$, we define the graded $S_n$-module $P(l)=\oplus_{i \in \bZ}P(l)_i$ by $P(l)_i:=P_{i+l}$. 
Then we obtain the degree shift functors $(l):\mathrm{grfr}(S) \xrightarrow{\sim} \mathrm{grfr}(S)$ and  $(l): \MF(W) \xrightarrow{\sim} \MF(W)$.  We put $\tau:=[(1)]:\HMF(W) \xrightarrow{\sim}\HMF(W)$.

\subsection{Matrix factorizations and Landau-Ginzburg models}

Fix a positive integer $d$ and a homogeneous polynomial $W \in S_n$ of degree $d$.
Consider the action of $\bC^*$ on the affine space $\bC^{n}$ defined by
\[\lambda\cdot (x_1, \cdots, x_n):=(\lambda x_1, \cdots, \lambda x_n) \]
for $\lambda \in \bC^*$ and $(x_1, \cdots, x_n) \in \bC^n$. 
Let $\chi_d:\bC^* \to \bC^*$ be the character of $\bC^*$ defined by  $\chi_d(t):=t^d$.
Then $W:\bC^n \to \bC$ is a $\chi_d$-semi-invariant regular function. 
The data $(\bC^n, \chi_d,W)^{\bC^*}$ is a gauged Landau-Ginzburg model.

Denote the category of $\bC^*$-equivariant locally free coherent sheaves on $\bC^n$ by $\vect_{\bC^*}(\bC^n)$. Taking global sections, we have the equivalence 
$\Gamma:\vect_{\bC^*}(\bC^n) \xrightarrow{\sim}\mathrm{grfr}(S_n)$. 
Note that, for $E \in \vect_{\bC^*}(\bC^n)$, 
the $S_n$-module $\Gamma(E)$ has the structure of a graded $S_n$-module induced by the weight decomposition with respect to the induced action of $\bC^*$ on $\Gamma(E)$. 
This induces the equivalence
$\Gamma: \vect_{\bC^*}(\bC^n, \chi_d,W) \xrightarrow{\sim} \MF(W)$ of dg categories.
Hence by Proposition \ref{prop:quasi-eq}, we have the quasi-equivalence $\MF(W) \to \Dabs\fact_{\bC^*}(\bC^n,\chi_d,W)$.
Taking homotopy categories, we obtain the equivalence 
\[\HMF(W) \xrightarrow{\sim} \Dabs[\fact_{\bC^*}(\bC^n,\chi_d,W)].\]

\begin{rmk}
By \cite[Theorem 40]{orl09}  and \cite[Lemma 4.8]{BDFIK18}, the triangulated category $\HMF(W)$ is idempotent complete.
\end{rmk}

\subsection{Serre functor}

Fix a positive integer $d$ and a homogeneous polynomial $W \in S_n$ of degree $d$.
The Serre functor of $\HMF(W)$ is described as follows.
\begin{thm}[{\cite[Theorem 3.8]{KST09}}]\label{thm:Serre}
The Serre functor $S$ of $\HMF(W)$ is given by $S:=\tau^{d-n}[n-2]$.
\end{thm}

Using Theorem \ref{thm:Serre}, we describe the Serre functor of $\Dabs[\fact_G(\bC^n, \chi_d,W)]$. 
Note that the following diagram commutes:
\begin{equation} \label{degree shift2}
\xymatrix{
&\vect_{\bC^*}(\bC^n, \chi_d,W) \ar[r]^-\Gamma 
\ar[d]_-{-\otimes \mcO_{\bC^n}(\chi_1)} 
&\MF(W) \ar[d]^-{(1)} \\
&\vect_{\bC^*}(\bC^n, \chi_d, W) \ar[r]_-\Gamma &\MF(W). 
}
\end{equation}

Taking homotopy categories, we obtain the commutative diagram:
\begin{equation} \label{degree shift3}
\xymatrix{
&D^{\abs}[\fact_{\bC^*}(\bC^n,\chi_d,W)] \ar[r]^{\ \ \ \ \ \ \  \sim}
\ar[d]_-{-\otimes \mcO_{\bC^n}(\chi_1)} 
&\HMF(W) \ar[d]^-{\tau} \\
&D^{\abs}[\fact_{\bC^*}(\bC^n,\chi_d,W)] \ar[r]^{\ \ \ \ \ \ \  \sim} &\HMF(W). 
}
\end{equation}

By Theorem \ref{thm:Serre} and \eqref{degree shift3}, we have the following.
\begin{thm}[{\cite[Theorem 3.8]{KST09}}]
The functor
\[-\otimes \mcO_{\bC^n}(\chi_{d-n})[n-2]:D^{\abs}[\fact_{\bC^*}(\bC^n,\chi_d,W)]\xrightarrow{\sim} D^{\abs}[\fact_{\bC^*}(\bC^n,\chi_d,W)]\] is the Serre functor.
\end{thm}

\section{Construction of perverse schobers} \label{sec:construct}
Let $n \geq 1$ be a positive integer, 
$W \in \bC[x_1, \cdots, x_{n+2}]$ be 
a homogeneous polynomial of degree $n+2$. 
We take the polynomial $W$ general so that 
the hypersurface 
$X \coloneqq (W=0) \subset \bP^{n+1}$ 
is smooth, which is a projective Calabi--Yau variety 
of dimension $n$. 
We denote by $\mcO_X(1) \coloneqq \mcO_{\bP^{n+1}}(1)|_X$. 

In this section, we construct a perverse schober 
using the following Orlov's theorem: 

\begin{thm}[\cite{orl09}, {\cite[Proposition 5.8]{bfk12}}]
\label{thm:Orlov1}
There exists an equivalence 
\begin{equation}
\psi \colon D^b(X) \xrightarrow{\sim} \HMF(W) 
\end{equation}
between triangulated categories. 
Moreover, the following diagram commutes: 
\begin{equation} \label{eq:Orlov1}
\xymatrix{
&D^b(X) \ar[r]^-\psi 
\ar[d]_-{\ST_{\mcO_X} \circ (- \otimes \mcO_X(1))} 
&\HMF(W) \ar[d]^-\tau \\
&D^b(X) \ar[r]_-\psi &\HMF(W). 
}
\end{equation}
\end{thm}

\begin{rmk}
More precisely, we have equivalences 
\[
\psi_w \colon D^b(X) \to \HMF(W) 
\]
indexed by integers $w \in \bZ$, 
and we have the commutative diagram (\ref{eq:Orlov1}) 
for a particular choice of $w \in \bZ$. 
See Section \ref{sec:vgit} for more detail. 
\end{rmk}

Take a disk 
$\Delta \subset \bP^1 \setminus \{0, \infty\}$ 
containing the point $1$. 
We fix two distinct points 
$x^\pm \in \Delta \setminus \{1\}$. 
First we construct 
a $\{x^\pm\}$-coordinatized local system of categories 
on $\bP^1 \setminus \{0, 1, \infty \}$. 

Let $a, b \in \pi_1(\bP^1 \setminus \{0, 1, \infty \}, x^-)$ 
be simple loops around $\infty, 1$, respectively. 
Let $\gamma$ be a path from $x^-$ to $x^+$ 
which is contained in $\Delta$ and does not 
go around the point $1 \in \Delta$. 
See Figure \ref{fig:generator} below. 
Then the groupoid 
$\pi_1(\bP^1 \setminus \{0, 1, \infty \}, \{x^\pm\})$ 
is freely generated by the three paths $a, b, \gamma$.

\begin{figure}[htbp]
    \centering
    \begin{tikzpicture}
  \shade[ball color = gray!40, opacity = 0.4] (0,0) circle (5cm);
  \draw (0,0) circle (5cm);
  \draw (-5,0) arc (180:360:5 and 1.0);
  \draw[dashed] (5,0) arc (0:180:5 and 1.0);
  \draw (0, -1) circle (0.15cm) node[left]{$1~$}; 
  \fill[fill=white] (0, -1) circle (0.14cm); 
  \draw (0, 5) circle (0.15cm); 
  \draw (0, 4.9) node[below]{$\infty$}; 
  \fill[fill=white] (0, 5) circle (0.14cm); 
  \draw (0, -5) circle (0.15cm); 
  \fill[fill=white] (0, -5) circle (0.14cm); 
  \draw (0, -5.35) node{$0$}; 
  \fill[fill=black] (-2.5, -0.87) circle (0.15cm); 
  \fill[fill=black] (2.5, -0.87) circle (0.15cm); 
  \draw[thick, ->] (-2.3, -0.8) to 
  [out=45, in=135] (2.3, -0.8); 
  \draw[thick] (-2.5, -0.6) to 
  [out=90, in=180] (0, 5.2); 
  \draw[thick, ->] (0, 5.2) to 
  [out=0, in=40] (-2.4, -0.6); 
  \draw[thick] (-2.2, -0.85) to 
  [out=10, in=90] (0.5, -1); 
  \draw[thick, ->] (0.5, -1) to 
  [out=270, in=340] (-2.2, -1); 
  \draw (0, 0.1) node[above]{$\gamma$}; 
  \draw (-2.8, -0.9) node[below]{$x^-$}; 
  \draw (2.8, -0.95) node[below]{$x^+$}; 
  \draw (0.6, 5.2) node[right]{$a$}; 
  \draw (0, -1.6) node[below]{$b$}; 
\end{tikzpicture}
    \caption{Generators of the groupoid 
    $\pi_1(\bP^1 \setminus \{0, 1, \infty \}, \{x^\pm\})$.}
    \label{fig:generator}
\end{figure}
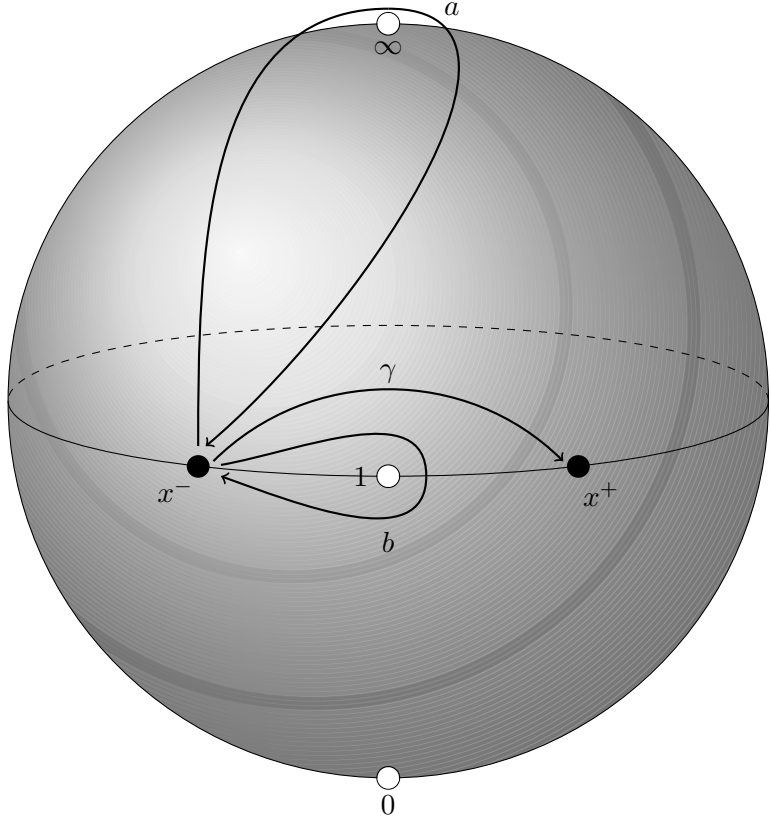

We first prepare the following two lemmas: 

\begin{lem} \label{lem:x-Loc}
There is a $\{x^-\}$-coordinatized local system $\fK$ 
on $\bP^1 \setminus \{0, 1, \infty\}$ with assignments 
\begin{align*}
    &x_- \mapsto D^b(X), \\
    &a \mapsto (-) \otimes \mcO_X(1), \quad 
    b \mapsto \ST_{\mcO_X}. 
\end{align*}
\end{lem}
\begin{proof}
Consider a homomorphism 
\begin{equation} \label{eq:homfree}
G \coloneqq \pi_1(\bP^1 \setminus \{0, 1, \infty\}, \{x^-\}) 
\to \Auteq D^b(X), 
\end{equation}
which sends the generators $a, b$ to 
the autoequivalences $(-) \otimes \mcO_X(1)$, 
$\ST_{\mcO_X}$, respectively. 

By \cite[Theorem 2.1 (1)]{bo20} 
(see also Section 2.3 in the same paper), 
there exists an obstruction  class 
$o \in H^3(G, \bC^*)$ 
for lifting the homomorphism (\ref{eq:homfree}) 
to an action of $G$ on the category $D^b(X)$. 
Since $G=\pi_1(\bP^1 \setminus \{0, 1, \infty\}, \{x^-\})$ 
is a free group, we have the vanishing 
$H^3(G, \bC^*)=0$, and hence the assertion holds. 
\end{proof}

\begin{rmk}
Although \cite[Theorem 2.1]{bo20} is 
stated only for finite groups, 
its proof works for infinite groups 
without any modifications. 
\end{rmk}

\begin{lem}[{\cite[Lemma 2.2]{bo20}}] \label{lem:commNTs}
Let $f, f' \colon \mcB \to \mcC$, 
$g, g' \colon \mcA \to \mcB$ be functors, 
and let $\alpha \colon f \to f'$, $\beta \colon g \to g'$ 
be natural transforms. 
Then the following diagram commutes: 
\[
\xymatrix{
&f \circ g \ar[r]^{\alpha g} \ar[d]_{f\beta} 
&f' \circ g \ar[d]^{f' \beta} \\
&f \circ g' \ar[r]_{\alpha g'} &f' \circ g'. 
}
\]
\end{lem}

\begin{prop} \label{prop:Loc}
There is an $\{x^\pm\}$-coordinatized local system $\fL$ 
on $\bP^1 \setminus \{0, 1, \infty \}$ with assignments 
\begin{equation}
\begin{aligned}
  &x^- \mapsto D^b(X), \quad x^+ \mapsto \HMF(W), \\
  &a \mapsto (-) \otimes \mcO_X(1), \quad 
  b \mapsto \ST_{\mcO_X}, \quad 
  \gamma \mapsto \psi. 
\end{aligned}
\end{equation}
\end{prop}
\begin{proof}

Let $\fK$ be the local system of categories 
constructed in Lemma \ref{lem:x-Loc}. 
By definition, it consists of the data
\[
\mu_{g'} \colon D^b(X) \xrightarrow{\sim} D^b(X), \quad 
\nu_{g', h'} \colon \mu_{g'} \mu_{h'} 
\xrightarrow{\sim} \mu_{g'h'}
\]
for all elements 
$g', h' \in \pi_1(\bP^1 \setminus \{0, 1, \infty \}, \{x^-\})$. 
Moreover, we have the following commutative diagrams 
\begin{equation} \label{eq:x-Loccomm}
\xymatrix{
&\mu_{g'} \mu_{h'} \mu_{k'} 
\ar[rr]^{\mu_{g'} \nu_{h', k'}} 
\ar[d]_{\mu_{g', h'} \mu_{k'}} 
&
&\mu_{g'} \mu_{h'k'} \ar[d]^{\nu_{g', h'k'}} \\
&\mu_{g'h'} \mu_{k'} \ar[rr]_{\nu_{g'h', k'}} 
&
&\mu_{g'h'k'}
}
\end{equation}
for all $g', h', k' \in 
\pi_1(\bP^1 \setminus \{0, 1, \infty \}, \{x^-\})$. 

We extend this local system $\fK$ to 
an $\{x^\pm\}$-coordinatized local system $\fL$. 
First observe that any morphism $g$ in the groupoid 
$\pi_1(\bP^1 \setminus \{0, 1, \infty \}, \{x^\pm\})$ 
can be uniquely written as follows: 
\begin{equation} \label{eq:redword}
    g= \gamma^{\epsilon} g' \gamma^{\delta}, 
    \quad \epsilon \in \{0, 1\}, \delta \in \{0, -1\}, 
    g' \in \pi_1(\bP^1 \setminus \{0, 1, \infty\}, \{x^-\}).
\end{equation}
To the morphism (\ref{eq:redword}), we associate the functor 
\[
\rho_g\coloneqq \psi^\epsilon \circ \mu_{g'} \circ \psi^\delta. 
\]
For each pair $g, h$ of composable morphisms, 
we associate a natural transform 
$\theta_{g, h} \colon \rho_g\rho_h \to \rho_{gh}$ as follows: 
\begin{itemize}
    \item When $g=\gamma, h=\gamma^{-1}$, 
    then we define $\theta_{\gamma, \gamma^{-1}}$ 
    to be the adjoint counit 
    \[
    \eta_{\psi^{-1}} \colon \psi \circ \psi^{-1} 
    \to \id_{\HMF(W)}. 
    \]
    \item When we have $g=\gamma^{\epsilon}g'\gamma^{-1}$, 
    $h=\gamma h'\gamma^\delta$ 
    for some $\epsilon, -\delta \in \{0, 1\}$, 
    $g', h' \in 
    \pi_1(\bP^1 \setminus \{0, 1, \infty\}, \{x^-\})$, 
    then we define $\theta_{g, h}$ as 
    \[
    \theta_{g, h} \coloneqq \nu_{g', h'} \eta_{\psi} \colon 
    \rho_g \rho_{h} \to \rho_{gh}, 
    \]
    where 
    $\eta_\psi \colon \psi^{-1} \circ \psi \to \id_{D^b(X)}$ 
    is the adjoint counit. 
    \item Otherwise, we put 
    $\theta_{g, h} \coloneqq \nu_{g', h'}$. 
\end{itemize}

We need to check the commutativity of 
the diagram (\ref{eq:Loccomm}) for 
all composable paths $g, h, k$. 
We prove it only for the case when we have 
$h=\gamma h'$ with 
$h' \in \pi_1(\bP^1 \setminus \{0, 1, \infty\}, \{x^-\})$. 
The other cases can be proved in a similar way. 
As in (\ref{eq:redword}), 
we write 
\[
g=\gamma^\epsilon g' \gamma^{-1}, \quad 
k=k' \gamma^\delta. 
\] 

Consider the following diagram: 
\[
\xymatrix{
&(\psi^\epsilon \mu_{g'} \psi^{-1})
(\psi \mu_{h'})(\mu_{k'}\psi^\delta) 
\ar[rr]^{\theta_{g, h}} 
\ar[rd]_{\eta_{\psi}} \ar[ddd]_{\theta_{h, k}=\nu_{h', k'}}
& &(\psi^\epsilon \mu_{g'h'})(\mu_{k'}\psi^{\delta}) 
\ar[ddd]^{\theta_{gh, k}=\nu_{g'h', k'}} \\
& &(\psi^\epsilon \mu_{g'})\mu_{h'}(\mu_{k'}\psi^{\delta}) 
\ar[ru]_{\nu_{g', k'}} \ar[d]_{\nu_{h', k'}}
& \\
& &(\psi^\epsilon \mu_{g'})(\mu_{h'k'}\psi^\delta) 
\ar[rd]^{\nu_{g', h'k'}}
& \\
&(\psi^\epsilon \mu_{g'} \psi^{-1})(\psi \mu_{h'k'}\psi^\delta) 
\ar[rr]_{\theta_{g, hk}} \ar[ru]^{\eta_{\psi}}
& &\psi^{\epsilon}\mu_{g'h'k'}\psi^\delta.
}
\]
The upper and the lower triangles are commutative 
by the definitions of $\theta_{g, h}$ and $\theta_{g, hk}$; 
the left square is commutative by Lemma \ref{lem:commNTs}; 
the right square is commutative 
by the commutativity of (\ref{eq:x-Loccomm}). 
Hence we conclude that the whole diagram commutes as required. 
\end{proof}

Next we shall extend the local system $\fL$ of categories 
to a perverse schober on $\bP^1$. 
Let us take a smooth hyperplane section $C \in |\mcO_X(1)|$, 
and denote by $i \colon C \hookrightarrow X$ the natural inclusion. 
We also take a general linear section 
$j \colon \bC^{n+1} \hookrightarrow \bC^{n+2}$. 
Note that by using the equivalences (see (\ref{degree shift3}))
\begin{align*}
&\HMF(W) \simeq 
D^{\abs}\left[\fact_{\bC^*}(\bC^{n+2}, \chi_{n+2}, W) \right], \\
&\HMF(W|_{\bC^{n+1}}) \simeq 
D^{\abs}\left[
\fact_{\bC^*}(\bC^{n+1}, \chi_{n+2}, W|_{\bC^{n+1}}) 
\right], 
\end{align*}
we have the natural push-forward functor 
$j_* \colon \HMF(W|_{\bC^{n+1}}) \to 
    \HMF(W)$. 

\begin{thm} \label{thm:schober1}
There exists a perverse schober $\fP$ on 
$(\bP^1, \{0, 1, \infty \})$ 
with the following properties: 
\begin{enumerate}
    \item The restriction of the schober $\fP$ to 
    $\bP^1 \setminus \{0, 1, \infty \}$ 
    coincides with the local system $\fL$ 
    of categories constructed in Proposition \ref{prop:Loc}, 
    \item Restricting the schober $\fP$ 
    to small disks around $\infty, 1, 0$, 
    we obtain the spherical functors and the signs 
    \begin{equation} \label{eq:sph}
    \begin{aligned}
    &i_* \colon D^b(C) \to D^b(X), \quad 
    \epsilon_\infty=1, \\
    &D^b(\pt) \to D^b(X), \quad \bC \mapsto \mcO_X, 
    \quad \epsilon_1=1,  \\
    &\psi^{-1} \circ j_* \colon 
    \HMF(W|_{\bC^{n+1}}) \to 
    \HMF(W) \xrightarrow{\sim} D^b(X), \quad 
    \epsilon_0=-1, 
    \end{aligned}
    \end{equation}
    respectively. 
\end{enumerate}
\end{thm}

\begin{rmk}
Our choices of spherical functors (\ref{eq:sph}) 
are compatible with the general result 
in Theorem \ref{thm:compsph}. 
Indeed, by \cite{orl09}, we have an SOD 
\[
\HMF(W|_{\bC^{n+1}})=\langle 
D^b(C), D^b(\pt)
\rangle, 
\]
and the dual twist of the spherical functor $\psi^{-1} \circ j_*$ 
is exactly the composition 
$(- \otimes \mcO_X(1))^{-1} \circ \ST_{\mcO_X}^{-1}$. 

Our point is that we have explicit descriptions of 
the source categories of spherical functors 
in terms of derived categories and 
matrix factorization categories, 
rather than abstract dg categories. 
\end{rmk}

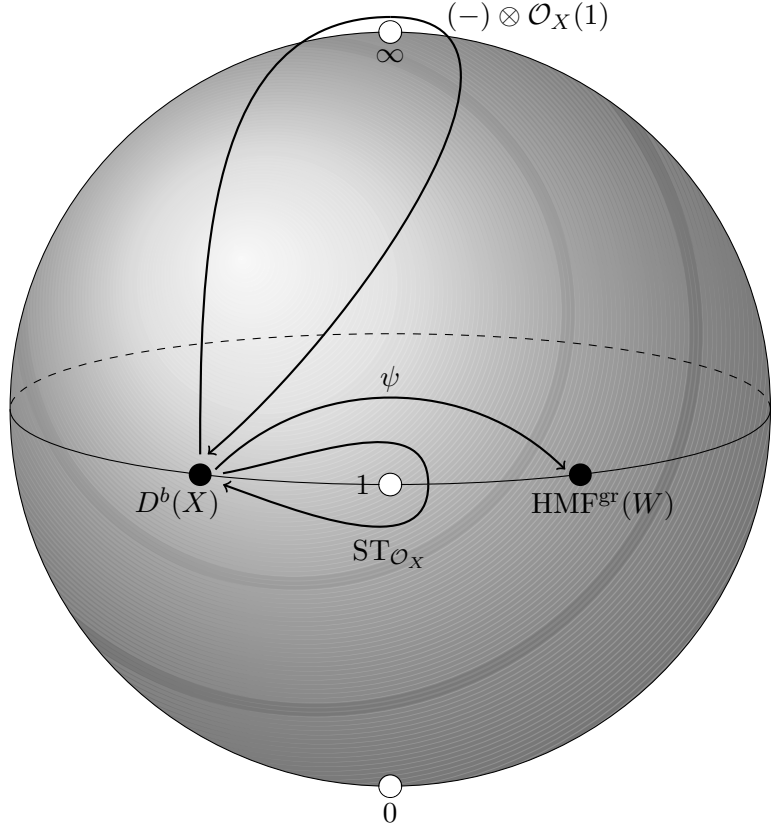
\begin{figure}[htbp]
    \centering
    \begin{tikzpicture}
  \shade[ball color = gray!40, opacity = 0.4] (0,0) circle (5cm);
  \draw (0,0) circle (5cm);
  \draw (-5,0) arc (180:360:5 and 1.0);
  \draw[dashed] (5,0) arc (0:180:5 and 1.0);
  \draw (0, -1) circle (0.15cm) node[left]{$1~$}; 
  \fill[fill=white] (0, -1) circle (0.14cm); 
  \draw (0, 5) circle (0.15cm); 
  \draw (0, 4.9) node[below]{$\infty$}; 
  \fill[fill=white] (0, 5) circle (0.14cm); 
  \draw (0, -5) circle (0.15cm); 
  \fill[fill=white] (0, -5) circle (0.14cm); 
  \draw (0, -5.35) node{$0$}; 
  \fill[fill=black] (-2.5, -0.87) circle (0.15cm); 
  \fill[fill=black] (2.5, -0.87) circle (0.15cm); 
  \draw[thick, ->] (-2.3, -0.8) to 
  [out=45, in=135] (2.3, -0.8); 
  \draw[thick] (-2.5, -0.6) to 
  [out=90, in=180] (0, 5.2); 
  \draw[thick, ->] (0, 5.2) to 
  [out=0, in=40] (-2.4, -0.6); 
  \draw[thick] (-2.2, -0.85) to 
  [out=10, in=90] (0.5, -1); 
  \draw[thick, ->] (0.5, -1) to 
  [out=270, in=340] (-2.2, -1); 
  \draw (0, 0.1) node[above]{$\psi$}; 
  \draw (-2.8, -0.9) node[below]{$D^b(X)$}; 
  \draw (2.8, -0.95) node[below]{$\HMF(W)$}; 
  \draw (0.6, 5.2) node[right]{$(-) \otimes \mcO_X(1)$}; 
  \draw (0, -1.6) node[below]{$\ST_{\mcO_X}$}; 
\end{tikzpicture}
    \caption{Local system of categories.}
    \label{fig:groupoid}
\end{figure}

\begin{proof}[Proof of Theorem \ref{thm:schober1}]
We apply Proposition \ref{prop:extension} 
to the local system $\fL$ three times. 
It is enough to show that the functors (\ref{eq:sph}) are 
spherical and their twists coincide with the autoequivalences 
\[
(-) \otimes \mcO_X(1), \quad \ST_{\mcO_X}, \quad 
\ST_{\mcO_X} \circ ((-) \otimes \mcO_X(1)), 
\] 
respectively (cf. Remark \ref{rmk:dual=inverse} and 
Example \ref{ex:trivialmonod}). 

We only prove the assertion for the functor 
\begin{equation} \label{eq:pushMF}
\psi^{-1} \circ j_* \colon 
\HMF(W|_{\bC^{n+1}}) \to D^b(X), 
\end{equation}
namely, we show that it lifts to a spherical dg functor 
and the corresponding twist autoequivalence of $D^b(X)$ 
is isomorphic to 
$\ST_{\mcO_X} \circ (- \otimes \mcO_X(1))$. 
We take dg enhancements of the categories 
$\HMF(W)$ and $\HMF(W|_{\bC^{n+1}})$ 
to be 
$D^{\abs}\fact_{\bC^*}(\bC^{n+2}, \chi_{n+2}, W)$ and 
$D^{\abs}\fact_{\bC^*}(\bC^{n+1}, \chi_{n+2}, W|_{\bC^{n+1}})$, 
respectively. 
Note that they also give Morita enhancements, 
since the categories $D^b(X)$ and 
$\HMF(W|_{\bC^{n+1}})$ are idempotent complete. 
Since $j \colon \bC^{n+1} \hookrightarrow \bC^{n+2}$ 
is a closed embedding, 
we have the functor 
\[
j_* \colon 
D^{\abs}\fact_{\bC^*}(\bC^{n+1}, \chi_{n+2}, W|_{\bC^{n+1}}) \to 
D^{\abs}\fact_{\bC^*}(\bC^{n+2}, \chi_{n+2}, W), 
\]
which gives a dg enhancement of the functor (\ref{eq:pushMF}). 

Next we construct dg lifts of the left and right adjoint functors 
\[
j^* \circ \psi,~ j^! \circ \psi \colon 
D^b(X) \to \HMF(W|_{\bC^{n+1}}). 
\]
For this we take dg enhancements of 
$\HMF(W)$ and 
$\HMF(W|_{\bC^{n+1}})$ as 
$\vect_{\bC^*}(\bC^{n+2}, \chi_{n+2}, W)$ and 
$\vect_{\bC^*}(\bC^{n+1}, \chi_{n+2}, W|_{\bC^{n+1}})$, 
respectively. 
Then we have the dg enhancement 
\[
j^* \colon \vect_{\bC^*}(\bC^{n+2}, \chi_{n+2}, W) 
\to \vect_{\bC^*}(\bC^{n+1}, \chi_{n+2}, W|_{\bC^{n+1}}). 
\]
Similarly, we have the enhancement of the functor 
$j^!=j^*(-) \otimes \mcO(\chi_1)[-1]$, 
where $\mcO(\chi_1)$ is 
the $\bC^*$-equivariant line bundle of weight $1$. 
Note that the dg enhancements 
$D^{\abs}\fact(-)$ and $\vect(-)$ are quasi-equivalent 
(see Proposition \ref{prop:quasi-eq}).

Let us consider the twist $T$ of $j_*$, 
which fits into the exact triangle 
\begin{equation} \label{eq:trisph}
j_* \circ j^! \to \id \to T
\end{equation}
of quasi-functors on 
$D^{\abs}\fact_{\bC^*}(\bC^{n+2}, \chi_{n+2}, W)$. 
Applying the underlying exact functors of (\ref{eq:trisph}) 
to an object $E \in \HMF(W)$, 
we obtain a triangle 
\begin{equation} \label{eq:degtri}
E(\chi_1)|_{\bC^{n+1}}[-1] \to E \to T(E)
\end{equation}
in $\HMF(W)$. 
The exact triangle (\ref{eq:degtri}) 
is isomorphic to the triangle 
obtained by tensoring $E$ to the sequence 
$\mcO(\chi_1)|_{\bC^{n+1}}[-1] \to \mcO \to \mcO(\chi_1)$. 
In other words, we have isomorphisms 
\begin{equation} \label{eq:degisom}
[T] \simeq (- \otimes \mcO(\chi_1)) \simeq \tau 
\end{equation}
between functors. 
In particular, the functor $[T]$ is an autoequivalence 
of $\HMF(W)$. 
Similarly, we can check that 
the endofunctor $[F]$ of 
$\HMF(W|_{\bC^{n+1}})$ 
is an equivalence, 
where $F$ is the cotwist of $j_*$ 
defined by the triangle 
\[
F \to \id \to j^! \circ j_*. 
\]
Indeed, we can show that 
\begin{equation} \label{eq:Serren+1}
    [F] \simeq S_{\HMF(W|_{\bC^{n+1}})}[-n-1], 
\end{equation}
where $S_{(-)}$ denotes the Serre functor. 

The above arguments show that 
the push-forward functor $j_*$ is a spherical functor. 
Moreover, its twist is isomorphic to the degree shift functor 
by (\ref{eq:degisom}). 
Via the equivalence 
$\psi \colon D^b(X) \simeq \HMF(W)$, 
the degree shift functor $\tau$ corresponds to 
the autoequivalence $\ST_{\mcO_X} \circ (- \otimes \mcO_X(1))$ 
as required (see Theorem \ref{thm:Orlov1}). 
\end{proof}

\begin{prop} \label{prop:CYschober}
The schober $\fP$ constructed in Theorem \ref{thm:schober1} 
is $n$-Calabi-Yau. 
\end{prop}
\begin{proof}
The local system $\fL$ clearly satisfies 
the Calabi-Yau property. 

By the isomorphism (\ref{eq:Serren+1}), 
the schober $\fP$ restricted to a disk around $0$ 
is $n$-Calabi-Yau. 
Similarly, one can check that 
it is $n$-Calabi-Yau around the points $\infty, 1$. 
\end{proof}

\section{Decategorification} \label{sec:decat}
In this section, we prove that 
the decategorifications of our perverse schobers 
coincide with the intersection complexes.

\subsection{Intersection complex}
We first recall the description of intersection complexes 
under the equivalence in Proposition \ref{prop:quiver}: 

\begin{prop}[{\cite[Proposition 3.2]{don19a}}] \label{prop:ICcrite}
Given a local system $L$ on $\Delta \setminus \{0\}$ 
with fiber $F$ and monodromy $m$, 
the intersection complex $\IC(L) \in \Perv(\Delta, 0)$ corresponds to 
the diagram 
\begin{equation} \label{eq:ICdiag}
\xymatrix{
&F/F^m \ar@<1ex>[r]^-{\id-m} &F \ar@<1ex>[l]^-u, 
}
\end{equation}
where $F^m \subset F$ denotes the $m$-invariant part, 
and $u \colon F \to F/F^m$ denotes the quotient map.  
\end{prop}

\begin{rmk}
In \cite{don19a}, Donovan uses a slightly different version 
of the quiver description of the category $\Perv(\Delta, 0)$. 
Namely, our category $\mcP_1$ is the category of data 
$(D, D_1, u, v)$ such that $\id-v \circ u$ is an isomorphism. 
On the other hand, Donovan \cite[Proposition 3.1]{don19a} 
considers the condition that $\id+v \circ u$ is an isomorphism. 
As a result, the map $F/F^m \to F$ in (\ref{eq:ICdiag}) 
is $\id-m$ in our case, 
while it is $m-\id$ in \cite[Proposition 3.2]{don19a}. 
\end{rmk}

The following lemma is useful for our purpose: 
\begin{lem} \label{lem:IClocal}
Let $X$ be a smooth projective variety, 
$U \subset X$ be an open subset, 
and $L$ be a local system on $U$. 

Given a perverse sheaf $P$ on $X$, 
the following conditions are equivalent: 
\begin{enumerate}
\item We have an isomorphism 
$P \simeq \IC(L)$. 

\item There exists an analytic open covering $X=\cup_i V_i$ 
such that for every $i$, we have an isomorphism 
$P|_{V_i} \simeq \IC(L|_{V_i \cap U}) \in \Perv(V_i)$. 
\end{enumerate}
\end{lem}
\begin{proof}
This is a well-known property, 
and follows from the characterization of 
the intersection complexes given in 
\cite[Definition 8.4.3 (2)]{max19}, 
together with the fact that 
the category of perverse sheaves forms 
a stack in analytic topology, 
see e.g., \cite[Remark 8.2.11]{max19}. 
\end{proof}

\subsection{Decategorification of perverse schobers}
In this subsection, we consider the decategorification 
of our schober when $n \geq 2$. 
Let $H \in |\mcO_{\bP^{n+2}}(1)|$ be the hyperplane class. 
We also denote by $H$ its restrictions to $X$ and $C$. 

\begin{defin}
\begin{enumerate}
    \item We define the {\it $H$-part} 
    $\Lambda_H(X)$ (resp. $\Lambda_H(C)$) 
    of the cohomology $H^*(X, \bQ)$ (resp. $H^*(C, \bQ)$) 
    to be the subring generated by the hyperplane class $H$. 
    
    \item We define the {\it H-part} 
    $\Lambda_H(W|_{\bC^{n+1}})$ 
    of the numerical Grothendieck group 
    $K_{\num}\left(\HMF(W|_{\bC^{n+1}}) \right)_\bQ$ 
    to be $\Lambda_H(C) \oplus \bQ$. 
\end{enumerate}
\end{defin}

\begin{rmk}
By \cite{orl09}, we have an SOD 
\[
\HMF(W|_{\bC^{n+1}})=\langle
D^b(C), D^b(\pt)
\rangle. 
\]
Hence we have a natural embedding 
\[
\Lambda_H(W|_{\bC^{n+1}}) \hookrightarrow 
K_{\num}\left(\HMF(W|_{\bC^{n+1}}) \right)_\bQ. 
\]
\end{rmk}

The proof of the following lemma is straightforward: 
\begin{def-lem} \label{lem:decat}
The following assertions hold: 
\begin{enumerate}
\item The local system $\fL$ of categories 
in Proposition \ref{prop:Loc} induces 
a local system on $\bP^1 \setminus \{0, 1, \infty\}$ 
whose fiber is $\Lambda_H(X)$. 
We denote it by $L$. 

\item The spherical functors with the signs in (\ref{eq:sph}) 
induce linear maps 
\begin{equation} \label{eq:decatsph}
\begin{aligned}
&i_* \colon \Lambda_H(C) \to \Lambda_H(X), 
\quad i^! \colon \Lambda_H(X) \to \Lambda_H(C), \\
&\bQ \to \Lambda_H(X), 
\quad \Lambda_H(X) \to \bQ, \\
&j_* \colon \Lambda_H(W|_{\bC^{n+1}}) 
\to \Lambda_H(X), \quad 
j^* \colon \Lambda_H(X) \to 
\Lambda_H(W|_{\bC^{n+1}}), 
\end{aligned}
\end{equation}
which define perverse sheaves on $(\Delta, 0)$. 

\item The local system $L$ in (1) and 
the data (\ref{eq:decatsph}) define 
a perverse sheaf $P$ on $\bP^1$. 
\end{enumerate}

We call $L, P$ the {\it decategorifications} of $\fL, \fP$, 
respectively. 
\end{def-lem}

\begin{thm} \label{thm:decatIC}
Suppose that $n \geq 2$. 
Let $L, P$ be the decategorifications of $\fL, \fP$ 
as in Definition-Lemma \ref{lem:decat}. 

Then we have an isomorphism 
$P \simeq \IC(L) \in \Perv(\bP^1)$. 
Moreover, the perverse sheaf $P$ is Verdier self-dual. 
\end{thm}
\begin{proof}
By Lemma \ref{lem:IClocal}, 
it is enough to check the isomorphism of the perverse sheaves 
$\IC(L)$ and $P$ in the neighborhoods 
of the points $0, 1, \infty$. 
To show this, we apply Proposition \ref{prop:ICcrite} 
to the data in (\ref{eq:decatsph}). 

We only prove the assertion for 
\begin{equation} \label{eq:sphtau}
\xymatrix{
&\Lambda_H(W|_{\bC^{n+1}}) \ar@<1ex>[r]^-{j_*} 
&\Lambda_H(X) \ar@<1ex>[l]^-{j^*} 
}, 
\end{equation}
namely, we prove that this is isomorphic to 
the diagram (\ref{eq:ICdiag}) with 
$F \coloneqq \Lambda_H(X)$ and 
$m \coloneqq 
(-\otimes \mcO_X(1))^{-1} \circ \ST_{\mcO_X}^{-1}$. 
We claim that $F^m=0$ in this case, 
and hence the diagram (\ref{eq:ICdiag}) becomes 
\begin{equation} \label{eq:ICtau}
\xymatrix{
&\Lambda_H(X) \ar@<1ex>[r]^-{\id-m} 
&\Lambda_H(X) \ar@<1ex>[l]^-\id. 
}
\end{equation}
To prove the claim, 
first note that $F^m=0$ if and only if 
$F^{(m^{-1})}=0$. 
Recall that we have the triangle 
\begin{equation} \label{eq:extau}
\dR\Hom(\mcO_X,- \otimes \mcO_X(1)) \otimes \mcO_X \to 
(-) \otimes \mcO_X(1) \to \ST_{\mcO_X}(- \otimes \mcO_X(1)). 
\end{equation}
Hence we have 
\[
m^{-1}(\vec{x})
=\vec{x}.e^H-\chi(\vec{x}.e^H)(1, 0, \cdots, 0)
\]
for any $\vec{x}=(x_0, \cdots, x_n) \in \Lambda_H(X)$. 
Assume that $\vec{x}$ is fixed by $m^{-1}$. 
Firstly, the second to the last components 
of $\vec{x}$ and $m^{-1}(\vec{x})$ 
must be equal, which is only possible when 
$x_0=x_1= \cdots =x_{n-1}=0$. 
Then the first component of $m^{-1}(\vec{x})$ is 
$-\chi(\vec{x}.e^H)=-x_n$, while we have seen $x_0=0$. 
Hence the $m$-fixed part is trivial as claimed. 

Next observe that we have the following morphism 
in the category $\mcP_1$ (cf. Definition \ref{def:catPn}): 
\begin{equation*} 
    \xymatrix{
    &\Lambda_H(X) \ar@<1ex>[r]^-{\id-m} \ar[d]_-{j^*}
    &\Lambda_H(X) \ar@<1ex>[l]^-\id \ar[d]^-\id \\
    &\Lambda_H(W|_{\bC^{n+1}}) 
    \ar@<1ex>[r]^-{j_*} 
    &\Lambda_H(X) \ar@<1ex>[l]^-{j^*}. 
    }
\end{equation*}
Hence, to prove that the diagrams 
(\ref{eq:sphtau}) and (\ref{eq:ICtau}) are isomorphic, 
it is enough to show that the map 
$j^* \colon \Lambda_H(X) \to 
\Lambda_H(\bC^{n+1}, W|_{\bC^{n+1}})$ is an isomorphism. 
Since we assume that $n \geq 2$, 
$C \subset X$ is a connected smooth variety, 
and the vector spaces 
$\Lambda_H(\bC^{n+1}, W|_{\bC^{n+1}})$ and 
$\Lambda_H(X)$ have the same dimension. 
Hence it is enough to show that 
$j_* j^*=\id-m$ is an isomorphism. 
Since we have proved 
\[
\ker(\id-m)=F^m=0, 
\]
the map $j_*j^*$ is an injective endomorphism, 
hence it is an isomorphism as required.



The Verdier self-duality of $P$ follows 
from Proposition \ref{prop:CYschober}, 
see also Remark \ref{rmk:Vdual}. 
\end{proof}

\begin{rmk}
When $n=1$, $X$ is an elliptic curve and 
$C$ consists of $3$ points. 
Then we would have 
$\dim \Lambda_H(W|_{\bC^2})=4 > 2=\dim \Lambda_H(X)$. 
Hence the above computation shows that 
$P$ and $\IC(L)$ are not isomorphic. 
This is related to the fact that $\mcO_X(1)$ has degree three, 
and hence does not generate the Picard group when $n=1$. 

We will modify the construction in the next section 
to get the categorification of the intersection complex 
for elliptic curves. 
\end{rmk}

\begin{rmk}
Even when $n \geq 2$, 
if we consider the full numerical $K$-groups 
instead of the $H$-parts, the result fails in general. 
Indeed, it may happen that 
$\dim K_{\num}(\HMF(W|_{\bC^{n+1}}))_\bQ 
> \dim K_{\num}(X)_\bQ$. 
\end{rmk}

\section{Mirror symmetry for elliptic curves} \label{sec:mirror}
In this section, we denote by $X$ a smooth elliptic curve. 
We construct perverse schobers in both $A$- and $B$-models, 
which are identified via the homological mirror symmetry 
for an elliptic curve. 

A key observation is that, in the case of elliptic curves, 
our schober only consists of spherical objects, 
while in higher dimension we have more general spherical functors.

\subsection{Perverse schober on the $B$-side}
The following is a straightforward modification 
of the construction in Theorem \ref{thm:schober1}: 

\begin{prop} \label{prop:schoberB}
Fix a point $p \in X$. 
There exists a perverse schober $\fP^B$ 
on $(\bP^1, \{0, 1, \infty\})$ satisfying the following properties: 
\begin{enumerate}
    \item The schober $\fP^B$ restricted to the open subset 
    $\bP^1 \setminus \{0, 1, \infty \}$ is isomorphic to 
    a local system $\fL^B$ of categories with the fiber $D^b(X)$, 
    defined by the assignments 
    \[
    a \mapsto (-) \otimes \mcO_X(p), \quad 
    b \mapsto \ST_{\mcO_X}.
    \]
    
    \item Restricting the schober $\fP^B$ to small disks 
    around $\infty, 1, 0$, we obtain the spherical functors 
    with the signs
    \begin{align*}
        &D^b(\pt) \to D^b(X), \quad \bC \mapsto \mcO_p, 
        \quad \epsilon_\infty=1, \\
        &D^b(\pt) \to D^b(X), \quad \bC \mapsto \mcO_X, 
        \quad \epsilon_1=1, \\
        &\langle D^b(\pt), D^b(\pt) \rangle \to D^b(X), 
        \quad \epsilon_0=-1, 
    \end{align*}
    respectively. 
    
    \item By taking the Grothendieck groups, 
    we obtain a local system $L^B$ and a perverse sheave $P^B$ 
    in the usual sense. We have $P^B \simeq \IC(L^B)$. 
\end{enumerate}
\end{prop}
\begin{proof}
The proof is almost identical to 
that of Theorem \ref{thm:schober1}. 
The first difference is that we use $(-) \otimes \mcO_X(p)$ 
instead of the degree three line bundle $\mcO_{\bP^2}(1)|_X$. 
The second difference is that 
the skyscraper sheaf $\mcO_p$ is spherical 
for an elliptic curve $X$, whose twist is isomorphic to 
$(-) \otimes \mcO_X(p)$. 

Finally, to construct a spherical functor 
$\langle D^b(\pt), D^b(\pt) \rangle \to D^b(X)$ 
whose dual twist is 
$(- \otimes \mcO_X(p))^{-1} \circ \ST_{\mcO_X}^{-1}$, 
we use Theorem \ref{thm:compsph}. 
\end{proof}

\subsection{Perverse schober on the $A$-side}
Here, we construct the {\it mirror} perverse schober 
using the homological mirror symmetry for an elliptic curve. 

Let $X=X/(\bZ \oplus \tau \bZ)$ be an elliptic curve, 
where $\tau$ is an element of the upper half plane. 
Then its {\it mirror} is defined to be a pair 
$X^\vee \coloneqq (T, \omega_\bC)$, 
where $T=\bC/\bZ^{\oplus 2}$ is a torus, 
and $\omega_\bC \coloneqq \tau dx \wedge dy$ 
is a {\it complexified K{\"a}hler form}. 
We denote by $\pi \colon \bC \to T$ the quotient map. 

We briefly recall the definition of 
the Fukaya category $D^\pi\Fuk(X^\vee)$, following \cite{pz98}. 
An object of the Fukaya category $D^\pi\Fuk(X^\vee)$ is isomorphic to a tuple $(L, \alpha, M)$, where 
\begin{itemize}
    \item $L \subset T$ is a special Lagrangian submanifold, 
    \item $\alpha \in \bR$ is a real number such that 
    the equation 
    \[
    L=\pi\left(\{ 
    z(t) \in \bC \colon z(t)=z_0+e^{i\pi\alpha}t, 
    \quad t \in \bR
    \}\right)
    \]
    holds for some $z_0 \in \bC$. 
    \item $M$ is a $\bC$-local system on $L$ 
    whose monodromy has eigenvalues in 
    the unit circle. 
\end{itemize}

Morphisms in the Fukaya category are defined to be 
the morphisms between local systems restricted to 
the intersection of Lagrangian submanifolds. 
Finally, the $A_\infty$-structure is defined by using 
the complexified K{\"a}hler form $\omega_\bC$ and 
the moduli spaces of pseudo-holomorphic disks in $X^\vee$. 

We denote by 
$\mcL_0=(L_0, 1/2, \underline{\bC}_{L_0})$ 
(resp. $\mcL_1=(L_1, 0, \underline{\bC}_{L_1})$) 
$\in D^\pi\Fuk(X^\vee)$, 
where $L_0$ (resp. $L_1$) is the image of the imaginary axis 
(resp. the real axis) in $\bC$ under the projection map 
$\pi \colon \bC \to T$. 

\begin{thm}[\cite{pz98}] \label{thm:mirrorell}
There exists an equivalence 
\begin{equation} \label{eq:mirrorell}
\Phi_{\mirror} \colon D^b(X) \xrightarrow{\sim} D^\pi\Fuk(X^\vee), 
\end{equation}
which sends $\mcO_e, \mcO_X$ to 
$\mcL_0, \mcL_1$, respectively, 
where $e \in X$ is the origin. 
\end{thm}

In particular, the objects 
$\mcL_i \in D^\pi\Fuk(X^\vee)$ ($i=0, 1$) 
are spherical objects. 
We denote by $\ST_{\mcL_i} \in \Auteq(D^\pi\Fuk(X^\vee))$ 
the corresponding spherical twists, 
which are called the {\it Dehn twists along $\mcL_i$}. 

Combining Proposition \ref{prop:schoberB} 
and Theorem \ref{thm:mirrorell}, 
we obtain the following: 

\begin{thm} \label{thm:schoberA}
There exists a perverse schober $\fP^A$ 
on $(\bP^1, \{0, 1, \infty\})$ satisfying the following properties: 
\begin{enumerate}
    \item The schober $\fP^B$ restricted to the open subset 
    $\bP^1 \setminus \{0, 1, \infty \}$ is isomorphic to 
    a local system $\fL^A$ of categories with the fiber 
    $D^\pi\Fuk(X^\vee)$, 
    defined by the assignments 
    \begin{equation} \label{eq:mirrorLC}
    a \mapsto \ST_{\mcL_0}, \quad 
    b \mapsto \ST_{\mcL_1}. 
    \end{equation}
    
    \item Restricting the schober $\fP^A$ to small disks 
    around $\infty, 1, 0$, we obtain the spherical functors 
    with the signs 
    \begin{align*}
        &D^b(\pt) \to D^\pi\Fuk(X^\vee), 
        \quad \bC \mapsto \mcL_0, 
        \quad \epsilon_\infty=1, \\
        &D^b(\pt) \to D^\pi\Fuk(X^\vee), 
        \quad \bC \mapsto \mcL_1, 
        \quad \epsilon_1=1, \\
        &\langle D^b(\pt), D^b(\pt) \rangle \to D^\pi\Fuk(X^\vee), 
        \quad \epsilon_0=-1, 
    \end{align*}
    respectively. 
\end{enumerate}

Under the mirror symmetry equivalence (\ref{eq:mirrorell}), 
the perverse schober $\fP^A$ is identified with 
the schober $\fP^B$ with $p=e \in X$. 
\end{thm}

\section{Orlov equivalences via VGIT} \label{sec:vgit}
In this section, we review the proof of Orlov's theorem 
(Theorem \ref{thm:Orlov1}) via the magic window theorem and 
Kn{\"o}rror periodicity, following \cite[Section 7]{bfk19}. 
This approach naturally gives an example of spherical pairs 
in the sense of Definition \ref{def:Spair}. 

\subsection{VGIT and window shift autoequivalences}
Let us put $G \coloneqq \bC^* \times \bC^*$, 
$Y:=\bC^{n+2} \times \bC$, and define an action of $G$ on $X$ 
as follows: 
\[
(\lambda, \mu) \cdot (\vec{x}, u) \coloneqq 
\left(\lambda \vec{x}, \lambda^{-(n+2)}\mu u \right), 
\quad (\lambda, \mu) \in G, (\vec{x}, u) \in Y. 
\]
Let $\chi^{\pm} \colon G \to \bC^*$ be characters defined as 
$\chi(\lambda, \mu) \coloneqq \lambda^\pm.$ 
By taking the GIT quotients with respect to 
the characters $\chi^\pm$, we obtain 
\[
Y^+ = [V(\mcO_{\bP^{n+1}}(-n-2))/\bC^*], \quad 
Y^- = [\bC^{n+2}/\bC^*], 
\]
where $\bC^*$ acts on the fiber of $V(\mcO_{\bP^{n+1}}(-n-2))$ 
by weight one, and $\bC^*$ acts on $\bC^{n+2}$ by weight one. 
We have natural embeddings 
$Y^\pm \hookrightarrow [Y/G]$. 

We further take a general homogeneous polynomial 
$W \in \bC[x_1, \cdots, x_{n+2}]$ of degree $n+2$, 
and define a function $Q_W \colon Y \to \bC$ by 
$Q_W(\vec{x}, u):=uW(\vec{x})$. 
Then $Q_W$ is a $\eta$-semi-invariant function, where 
$\eta \colon G \to \bC^*$ is a character defined as 
$\eta(\lambda, \mu)=\mu$. 
Indeed, we have 
\[
Q_W((\lambda, \mu) \cdot (\vec{x}, u)))
=(\lambda^{-n-2}\mu u) \lambda^{n+2}W(\vec{x})
=\mu u W(\vec{x})=\mu Q_W(\vec{x}, u). 
\]
Hence the data $(Y, \eta, Q_W)^G$ defines 
a gauged Landau-Ginzburg model. 

\begin{defin} \label{def:window}
\begin{enumerate}
\item For an interval $I$, we define 
the {\it window subcategory} 
$\mcG_I \subset D^\abs[\fact_G(Y, \eta, Q_W)]$ 
to be a triangulated subcategory generated 
by factorizations (\ref{eq:deffact}) with terms 
\[
E_j \simeq \bigoplus_{i \in I \cap \bZ}\mcO(i)^{l_{ij}}, \quad 
j=1, 0, \quad l_{ij} \in \bZ_{\geq 0} 
\]
as $\bC^*$-equivariant sheaves on $Y$,  
where $\bC^*$ acts on $Y$ via the inclusion $\bC^* \subset G$ 
into the first factor. 

\item For an integer $w \in \bZ$, we put 
\[
\mcG_w \coloneqq \mcG_{[w, w+n+2)}, \quad 
\overline{\mcG}_w \coloneqq \mcG_{[w, w+n+2]}. 
\]
\end{enumerate}
\end{defin}

The following is a special case of the main theorem of 
\cite{bfk19, hl15}: 
\begin{thm}[\cite{bfk19, hl15}] \label{thm:window}
For each integer $w \in \bZ$, the compositions 
\begin{align*}
&r^+_w \colon \mcG_w 
\hookrightarrow D^\abs[\fact_G(Y, \eta, Q_W)] 
\xrightarrow{\res^+} 
D^\abs[\fact_{\bC^*}(
V(\mcO_{\bP^{n+1}}(-n-2)), \chi_1, Q_W)], \\
&r^-_w \colon \mcG_w 
\hookrightarrow D^\abs[\fact_G(Y, \eta, Q_W)] 
\xrightarrow{\res^-} 
D^\abs[\fact_{\bC^*}(
\bC^{n+2}, \chi_{n+2}, W)], 
\end{align*}
are equivalences, 
where $\res^\pm$ denote the restriction functors to 
the semi-stable loci $Y^\pm$. 
In particular, we have an equivalence 
\begin{align*}
\psi_w \coloneqq 
r^-_w \circ (r^+_w)^{-1} \colon 
D^\abs[\fact_{\bC^*}(
&V(\mcO_{\bP^{n+1}}(-n-2)), \chi_1, Q_W)] \\
&\quad \to D^\abs[\fact_{\bC^*}(\bC^{n+2}, \chi_{n+2}, W)]. 
\end{align*}
\end{thm}

\begin{defin} \label{def:windowshift}
For each integer $w \in \bZ$, we put 
\[
\Phi_w \coloneqq \psi_{w-1}^{-1} \circ \psi_w 
\in \Auteq\left( 
D^\abs[\fact_{\bC^*}(
V(\mcO_{\bP^{n+1}}(-n-2)), \chi_1, Q_W)]
\right)
\]
and call it as the {\it window shift autoequivalence}. 
\end{defin}

In the following, 
we will interpret the window shift autoequivalences 
as autoequivalences of $D^b(X)$ under the Kn{\"o}rror periodicity, 
where $X \coloneqq (W=0) \subset \bP^{n+1}$ 
is a smooth Calabi-Yau hypersurface of dimension $n$. 
We have the following diagram: 
\begin{equation} \label{eq:Kndiag}
\xymatrix{
&X \ar@{^{(}->}[r]^{i_X} \ar@{^{(}->}[d]_{\gamma_X} 
&\bP^{n+1} \ar@{^{(}->}[d]^{\gamma} \\
&V(\mcO_{\bP^{n+1}}(-n-2))|_X \ar@{^{(}->}[r]^{i} \ar[d]_{p} 
&V(\mcO_{\bP^{n+1}}(-n-2)) \ar[r]^-{Q_W} \ar[d]^{q} &\bC \\
&X \ar@{^{(}->}[r]_{i_X} &\bP^{n+1}, 
}
\end{equation}
where $p, q$ denote the projections, 
$\gamma, \gamma_X$ are the inclusions as zero sections, 
and $i, i_X$ are the natural inclusions.

We use the following result: 
\begin{prop}[{\cite[Proposition 3.4]{hls16}}] \label{prop:sphfact}
The window shift autoequivalence 
$\Phi_{w}$ fits into the following exact triangle 
\begin{equation} \label{eq:sphfact}
\dR\Hom(\gamma_*\mcO_{\bP^{n+1}}(-w-n-1), -) 
\otimes \gamma_*\mcO_{\bP^{n+1}}(-w-n-1) 
\to \id \to \Phi_{w}. 
\end{equation} 
In other words, $\Phi_{w}$ is the spherical twist around 
$\gamma_*\mcO_{\bP^{n+1}}(-w-n-1)$. 
\end{prop}

Recall from Theorem \ref{thm:Knorrer} 
that we have an equivalence 
\begin{equation} \label{eq:KnHS}
i_*p^* \colon D^b(X) \xrightarrow{\sim} 
D^\abs[\fact_{\bC^*}(
V(\mcO_{\bP^{n+1}}(-n-2)), \chi_1, Q_W)]. 
\end{equation}

\begin{prop} \label{prop:window=st}
Let $w \in \bZ$ be an integer. 
Under the K{\"o}rror periodicity equivalence (\ref{eq:KnHS}), 
the window shift autoequivalence $\Phi_w$ 
corresponds to the spherical twist 
$\ST_{\mcO_X(w+2n+3)}$ on $D^b(X)$. 
\end{prop}
\begin{proof}
To simplify the notation, we put 
$V \coloneqq V(\mcO_{\bP^{n+1}}(-n-2))$. 
Since we have an isomorphism 
\[
\Phi_{k-n-1} \simeq ((-) \otimes \mcO_V(k)) \circ \Phi_{-n-1} 
\circ ((-) \otimes \mcO_V(-k))
\]
of functors (see e.g., \cite[Lemma 3.2]{hls16}) 
for each $k \in \bZ$, 
we may assume $w=-n-1$. 

We prove a functorial isomorphism 
$\Phi_{-n-1} \circ i_*p^*(E) 
\simeq i_*p^* \circ \ST_{\mcO_X(n+2)}(E)$ 
for each object $E \in D^b(X)$. 
By applying the functor $i_*p^*$ 
to the defining triangle 
\[
\dR\Hom(\mcO_X(n+2), E) \otimes \mcO_X(n+2) 
\to E \to \ST_{\mcO_X(n+2)}(E), 
\]
we obtain 
\begin{equation} \label{eq:sphDb}
    \dR\Hom(\mcO_X(n+2), E) \otimes i_*\mcO_{V|_X}(n+2) 
\to i_*p^*E \to i_*p^* \circ \ST_{\mcO_X(n+2)}(E). 
\end{equation}

On the other hand, by the triangle (\ref{eq:sphfact}), 
we have the exact triangle 
\begin{equation} \label{eq:phiip}
\dR\Hom(\gamma_*\mcO_{\bP^{n+1}}, i_*p^*E) 
\otimes \gamma_*\mcO_{\bP^{n+1}} 
\to i_*p^*E \to \Phi_{-n-1}(i_*p^*E). 
\end{equation}
We compute the complex 
$\dR\Hom(\gamma_*\mcO_{\bP^{n+1}}, i_*p^*E)$ as follows: 
\begin{equation} \label{eq:homadj}
\begin{aligned}
    \dR\Hom(\gamma_*\mcO_{\bP^{n+1}}, i_*p^*E) 
    &\simeq \dR\Hom(i^*\gamma_*\mcO_{\bP^{n+1}}, p^*E) \\
    &\simeq \dR\Hom(\gamma_{X*}i^*_X\mcO_{\bP^{n+1}}, p^*E) \\
&\simeq \dR\Hom(i_X^*\mcO_{\bP^{n+1}}, \gamma^*_Xp^*E(-n-2)[1]), \\
&\simeq \dR\Hom(\mcO_X(n+2), E)[1], 
\end{aligned}
\end{equation}
where the first isomorphism follows from the adjunction; 
the second isomorphism follows from 
Lemma \ref{lem:Koszzero} (2) below; 
the third isomorphism follows from the adjunction 
and the fact that $\omega_{X/V|_X} \simeq \mcO_X(-n-2)$; 
the last isomorphism follows from $p \circ \gamma_X=\id_X$. 

Moreover, we have an isomorphism 
\begin{equation} \label{eq:Kosz}
    i_*\mcO_{V|_X}(n+2)[-1] \simeq \gamma_*\mcO_{\bP^{n+1}}
\end{equation}
by Lemma \ref{lem:Koszzero} (1) below. 

By the isomorphisms (\ref{eq:homadj}) and (\ref{eq:Kosz}), 
it follows that the triangles 
(\ref{eq:sphDb}) and (\ref{eq:phiip}) are isomorphic, 
as required. 
\end{proof}

We have used the following lemma in the above proof: 
\begin{lem} \label{lem:Koszzero}
Put $\mcL \coloneqq \mcO_{\bP^{n+1}}(-n-2)$, 
$V \coloneqq V(\mcL)$.  
The following statements hold: 
\begin{enumerate}
    \item We have an isomorphism 
    \[
    \gamma_*\mcO_{\bP^{n+1}} 
    \simeq i_*\mcO_{V|_X}(n+2)[-1]. 
    \]
    \item We have an isomorphism  
    \[
    i^*\gamma_*\mcO_{\bP^{n+1}} \simeq \gamma_{X*}i^*_X\mcO_{\bP^{n+1}}. 
    \]
\end{enumerate}
\end{lem}
\begin{proof}
Put $s \coloneqq q^*W \colon q^*\mcL \to \mcO_V$, 
and let $t \colon \mcO_V \to q^*\mcL(\chi_1)$ 
be the tautological section. 
By Lemma \ref{lem:Koszul}, we have 
\[
\gamma_*\mcO_{\bP^{n+1}} \simeq K(t^\vee, s^\vee) 
\simeq K(s, t)^\vee \simeq i_*\mcO_{V|_X}(n_2)[-1], 
\]
which proves the first assertion. 

Similarly, we have 
\[
i^*\gamma_*\mcO_{\bP^{n+1}} 
\simeq i^*K(t^\vee, s^\vee) 
= K(t^\vee|_{V|_X}, 0) 
\simeq \gamma_{X*}i^*_X\mcO_{\bP^{n+1}}, 
\]
where the second equality follows from the vanishing $s|_{V|_X}=0$, 
hence the second assertion holds. 
\end{proof}

We can now give the proof of Orlov's theorem: 
\begin{proof}[Proof of Theorem \ref{thm:Orlov1}]
As before, we put 
$V \coloneqq V(\mcO_{\bP^{n+1}}(-n-2))$. 
The existence of the equivalence 
\[
D^b(X) \simeq \HMF(W) 
\]
follows from Theorem \ref{thm:window} 
together with the equivalences 
(\ref{eq:KnHS}) and (\ref{degree shift3}). 

For the second statement, 
consider the following commutative diagram: 
\begin{equation} \label{eq:winshif}
\xymatrix{
&D^\abs[\fact_{\bC^*}(V, \chi_1, Q_W)] 
\ar@/^2pc/[rr]^{\psi_{-2n-4}} \ar[d]_{\otimes \mcO(1)} 
&\mcG_{-n-3} \ar[l]_-{r^+_{-2n-4}} \ar[r]^-{r^-_{-2n-4}} 
\ar[d]_{\otimes \mcO(1)} 
&D^\abs[\fact_{\bC^*}(\bC^{n+2}, \chi_{n+2}, W)] 
\ar[d]_{\otimes \mcO(1)} \\
&D^\abs[\fact_{\bC^*}(V, \chi_1, Q_W)] 
\ar@/_2pc/[rr]_{\psi_{-2n-3}}
&\mcG_{-2n-3} \ar[l]^-{r^+_{-2n-3}} \ar[r]_-{r^-_{-2n-3}} 
&D^\abs[\fact_{\bC^*}(\bC^{n+2}, \chi_{n+2}, W)],
}
\end{equation}
where $\mcO(1) \in D^\abs[\fact_G(Y, \eta, Q_W)]$ denotes 
the $\bC^*$-equivariant line bundle of weight one, 
with respect to the $\bC^*$-action via the inclusion 
$\bC^* \subset G$ to the first factor. 

By (\ref{degree shift3}), the autoequivalence 
\[
(-)\otimes \mcO(1) \in 
\Auteq(
D^\abs[\fact_{\bC^*}(\bC^{n+2}, \chi_{n+2}, W)])
\]
corresponds to the degree shift equivalence $\tau$ 
under the natural equivalence 
\[
D^\abs[\fact_{\bC^*}(\bC^{n+2}, \chi_{n+2}, W)]) 
\simeq \HMF(W). 
\]
On the other hand, the autoequivalence $\otimes \mcO(1)$ 
on $D^\abs[\fact_{\bC^*}(V, \chi_1, Q_W)]$ 
corresponds to the autoequivalence $\otimes \mcO_X(1)$ 
on $D^b(X)$ via the Kn{\"o}rror periodicity (\ref{eq:KnHS}). 

By the above observations, 
the commutativity of the diagram (\ref{eq:winshif}) implies that 
\begin{align*}
(\psi_{-2n-4})^{-1} \circ \tau \circ \psi_{-2n-4}
&=(\psi_{-2n-4})^{-1} \circ \psi_{-2n-3} 
\circ (- \otimes \mcO_X(1)) \\
&=\Phi_{-2n-3} \circ (- \otimes \mcO_X(1)) \\ 
&=\ST_{\mcO_X} \circ (- \otimes \mcO_X(1)), 
\end{align*}
where the last equality follows from 
Proposition \ref{prop:window=st}. 
\end{proof}

\subsection{Spherical pairs from VGIT}
We end this section by constructing an example of spherical pairs 
using the theory of \cite{bfk19, hl15}, following \cite{don19a}. 
Let 
\[S^+ \coloneqq \{0\} \times \bC, 
S^- \coloneqq \bC^{n+2} \times \{0\} 
\subset Y=\bC^{n+2} \times \bC 
\]
be the unstable loci with respect to 
the characters $\chi^+, \chi^-$, respectively. 
Let $j^\pm \colon S^\pm \hookrightarrow Y$ 
denote the inclusions. 

\begin{prop}[\cite{don19a, hls16}]
We have a spherical pair 
\begin{equation} \label{eq:windowsphpair}
\overline{\mcG}_{-n-2}=\langle D^b(X), D^b(\pt) \rangle
=\langle \HMF(W), D^b(\pt) \rangle 
\end{equation}
such that the induced autoequivalence on $D^b(X)$ 
is isomorphic to $\ST_{\mcO_X}$. 
\end{prop}
\begin{proof}
By \cite[Equation (3)]{hls16}, 
applied to the two different KN stratifications 
\[
Y=(Y \setminus S^+) \sqcup S^+
=(Y \setminus S^-) \sqcup S^-,
\] 
we have a pair of semi-orthogonal decompositions: 
\begin{equation} \label{eq:SODbarG}
\overline{\mcG}_{-2n-4}=\langle 
\mcG_{-2n-4}, j^+_*\mcO_{S^+}(-2n-4)
\rangle
=\langle 
\mcG_{-2n-3}, j^-_*\mcO_{S^-}(-n-2) 
\rangle.
\end{equation}

The proofs of \cite[Theorem 4.4, Proposition 4.5]{don19a} show 
that the pair (\ref{eq:SODbarG}) defines a spherical pair. 
Moreover, by Theorem \ref{thm:window} and 
the equivalences (\ref{eq:KnHS}) and (\ref{degree shift3}), 
we have 
\[
\mcG_{-2n-4} \simeq D^b(X), \quad 
\mcG_{-2n-3} \simeq \HMF(W). 
\]
Hence we obtain the spherical pair 
as in (\ref{eq:windowsphpair}). 

It remains to compute the induced autoequivalence on $D^b(X)$. 
By construction, the induced autoequivalence on 
$D^\abs[\fact_{\bC^*}(V, \chi_1, Q_W)] \simeq \mcG_{-2n-4}$ is 
the spherical twist around the object 
$j^-_*\mcO_{S^-}(-n-2)|_{V} 
\simeq \gamma_*\mcO_{\bP^{n+1}}(-n-2)$. 
By Propositions \ref{prop:sphfact} and \ref{prop:window=st}, 
it corresponds to the spherical twist $\ST_{\mcO_X}$ 
under the Kn{\"o}rror periodicity equivalence (\ref{eq:KnHS}). 
\end{proof}

\begin{cor} \label{cor:upgrade}
The perverse schober in Theorem \ref{thm:schober1} 
upgrades to the spherical pair (\ref{eq:windowsphpair}) 
around the point $0 \in \bP^1$. 
\end{cor}

\bibliographystyle{alpha}
\bibliography{maths}

\newcommand{\etalchar}[1]{$^{#1}$}
\begin{thebibliography}{BDF{\etalchar{+}}18}

\bibitem[AL17]{al17}
R.~Anno and T.~Logvinenko.
\newblock Spherical {DG}-functors.
\newblock {\em J. Eur. Math. Soc. (JEMS)}, 19(9):2577--2656, 2017.

\bibitem[Bar20]{bar20}
Federico Barbacovi.
\newblock On the composition of two spherical twists, 2020.

\bibitem[BDF{\etalchar{+}}18]{BDFIK18}
M.~Ballard, D.~Deliu, D.~Favero, M.~U. Isik, and L.~Katzarkov.
\newblock On the derived categories of degree {$d$} hypersurface fibrations.
\newblock {\em Math. Ann.}, 371(1-2):337--370, 2018.

\bibitem[Bei87]{bei87}
A.~A. Beilinson.
\newblock How to glue perverse sheaves.
\newblock In {\em {$K$}-theory, arithmetic and geometry ({M}oscow,
  1984--1986)}, volume 1289 of {\em Lecture Notes in Math.}, pages 42--51.
  Springer, Berlin, 1987.

\bibitem[BFK12]{bfk12}
Matthew Ballard, David Favero, and Ludmil Katzarkov.
\newblock Orlov spectra: bounds and gaps.
\newblock {\em Invent. Math.}, 189(2):359--430, 2012.

\bibitem[BFK14]{bfk14}
M.~Ballard, D.~Favero, and L.~Katzarkov.
\newblock A category of kernels for equivariant factorizations and its
  implications for {H}odge theory.
\newblock {\em Publ. Math. Inst. Hautes \'{E}tudes Sci.}, 120:1--111, 2014.

\bibitem[BFK19]{bfk19}
M.~Ballard, D.~Favero, and L.~Katzarkov.
\newblock Variation of geometric invariant theory quotients and derived
  categories.
\newblock {\em J. Reine Angew. Math.}, 746:235--303, 2019.

\bibitem[BH06]{bh06}
L.~A. Borisov and R.~P. Horja.
\newblock Mellin-{B}arnes integrals as {F}ourier-{M}ukai transforms.
\newblock {\em Adv. Math.}, 207(2):876--927, 2006.

\bibitem[BKS18]{bks18}
A.~Bondal, M.~Kapranov, and V.~Schechtman.
\newblock Perverse schobers and birational geometry.
\newblock {\em Selecta Math. (N.S.)}, 24(1):85--143, 2018.

\bibitem[BO20]{bo20}
T.~Beckmann and G.~Oberdieck.
\newblock On equivariant derived categories, 2020.

\bibitem[CIR14]{cir14}
A.~Chiodo, H.~Iritani, and Y.~Ruan.
\newblock Landau-{G}inzburg/{C}alabi-{Y}au correspondence, global mirror
  symmetry and {O}rlov equivalence.
\newblock {\em Publ. Math. Inst. Hautes \'{E}tudes Sci.}, 119:127--216, 2014.

\bibitem[DK21]{dk21}
W.~Donovan and T.~Kuwagaki.
\newblock Mirror symmetry for perverse {S}chobers from birational geometry.
\newblock {\em Comm. Math. Phys.}, 381(2):453--490, 2021.

\bibitem[DKSS21]{dkss21}
T.~Dyckerhoff, M.~Kapranov, V.~Schechtman, and Y.~Soibelman.
\newblock Spherical adjunctions of stable $\infty$-categories and the relative
  s-construction, 2021.

\bibitem[Don19a]{don19a}
W.~Donovan.
\newblock Perverse {S}chobers and wall crossing.
\newblock {\em Int. Math. Res. Not. IMRN}, (18):5777--5810, 2019.

\bibitem[Don19b]{don19b}
W.~Donovan.
\newblock Perverse schobers on {R}iemann surfaces: constructions and examples.
\newblock {\em Eur. J. Math.}, 5(3):771--797, 2019.

\bibitem[GGM85]{ggm85}
A.~Galligo, M.~Granger, and Ph. Maisonobe.
\newblock {$D$}-modules et faisceaux pervers dont le support singulier est un
  croisement normal.
\newblock {\em Ann. Inst. Fourier (Grenoble)}, 35(1):1--48, 1985.

\bibitem[GMV96]{gmv96}
S.~Gelfand, R.~MacPherson, and K.~Vilonen.
\newblock Perverse sheaves and quivers.
\newblock {\em Duke Math. J.}, 83(3):621--643, 1996.

\bibitem[Hir17]{H17}
Y~Hirano.
\newblock Equivalences of derived factorization categories of gauged
  {L}andau-{G}inzburg models.
\newblock {\em Adv. Math.}, 306:200--278, 2017.

\bibitem[HL15]{hl15}
D.~Halpern-Leistner.
\newblock The derived category of a {GIT} quotient.
\newblock {\em J. Amer. Math. Soc.}, 28(3):871--912, 2015.

\bibitem[HLS16]{hls16}
D.~Halpern-Leistner and I.~Shipman.
\newblock Autoequivalences of derived categories via geometric invariant
  theory.
\newblock {\em Adv. Math.}, 303:1264--1299, 2016.

\bibitem[Hor99]{hor99}
R.~P. Horja.
\newblock {\em Hypergeometric functions and mirror symmetry in toric
  varieties}.
\newblock ProQuest LLC, Ann Arbor, MI, 1999.
\newblock Thesis (Ph.D.)--Duke University.

\bibitem[Isi13]{Isik}
Mehmet~Umut Isik.
\newblock Equivalence of the derived category of a variety with a singularity
  category.
\newblock {\em Int. Math. Res. Not. IMRN}, (12):2787--2808, 2013.

\bibitem[Kel94]{kel94}
B.~Keller.
\newblock Deriving {DG} categories.
\newblock {\em Ann. Sci. \'{E}cole Norm. Sup. (4)}, 27(1):63--102, 1994.

\bibitem[Kel06]{kel06}
B.~Keller.
\newblock On differential graded categories.
\newblock In {\em International {C}ongress of {M}athematicians. {V}ol. {II}},
  pages 151--190. Eur. Math. Soc., Z\"{u}rich, 2006.

\bibitem[KS15]{ks15}
M.~Kapranov and V.~Schechtman.
\newblock Perverse {S}chobers, 2015.

\bibitem[KSS20]{kss20}
M.~Kapranov, Y.~Soibelman, and L.~Soukhanov.
\newblock Perverse schobers and the {A}lgebra of the {I}nfrared, 2020.

\bibitem[KST09]{KST09}
H.~Kajiura, K.~Saito, and A.~Takahashi.
\newblock Triangulated categories of matrix factorizations for regular systems
  of weights with {$\epsilon=-1$}.
\newblock {\em Adv. Math.}, 220(5):1602--1654, 2009.

\bibitem[Max19]{max19}
L.~G. Maxim.
\newblock {\em Intersection homology \& perverse sheaves}, volume 281 of {\em
  Graduate Texts in Mathematics}.
\newblock Springer, Cham, [2019] \copyright 2019.
\newblock with applications to singularities.

\bibitem[Orl09]{orl09}
D.~Orlov.
\newblock Derived categories of coherent sheaves and triangulated categories of
  singularities.
\newblock In {\em Algebra, arithmetic, and geometry: in honor of {Y}u. {I}.
  {M}anin. {V}ol. {II}}, volume 270 of {\em Progr. Math.}, pages 503--531.
  Birkh\"{a}user Boston, Boston, MA, 2009.

\bibitem[PZ98]{pz98}
A.~Polishchuk and E.~Zaslow.
\newblock Categorical mirror symmetry: the elliptic curve.
\newblock {\em Adv. Theor. Math. Phys.}, 2(2):443--470, 1998.

\bibitem[SdB19]{svdb19}
S.~Spenko and M.~Van den Bergh.
\newblock A class of perverse schobers in {G}eometric {I}nvariant {T}heory,
  2019.

\bibitem[SdB20]{svdb20}
S.~Spenko and M.~Van den Bergh.
\newblock Perverse schobers and {GKZ} systems, 2020.

\bibitem[Sei15]{sei15}
P.~Seidel.
\newblock Homological mirror symmetry for the quartic surface.
\newblock {\em Mem. Amer. Math. Soc.}, 236(1116):vi+129, 2015.

\bibitem[She15]{she15}
N.~Sheridan.
\newblock Homological mirror symmetry for {C}alabi-{Y}au hypersurfaces in
  projective space.
\newblock {\em Invent. Math.}, 199(1):1--186, 2015.

\bibitem[Shi12]{S12}
Ian Shipman.
\newblock A geometric approach to {O}rlov's theorem.
\newblock {\em Compos. Math.}, 148(5):1365--1389, 2012.

\bibitem[To{\"e}07]{toe07}
B.~To{\"e}n.
\newblock The homotopy theory of {$dg$}-categories and derived {M}orita theory.
\newblock {\em Invent. Math.}, 167(3):615--667, 2007.

\bibitem[To{\"e}11]{toe11}
B.~To{\"e}n.
\newblock Lectures on dg-categories.
\newblock In {\em Topics in algebraic and topological {$K$}-theory}, volume
  2008 of {\em Lecture Notes in Math.}, pages 243--302. Springer, Berlin, 2011.

\end{thebibliography}

\end{document}